%% file: DA_NSC_prox_-_8.tex
\documentclass[11pt]{article}

\usepackage{fullpage}
\title{Dual Averaging With Non-Strongly-Convex Prox-Functions:\\ New Analysis and Algorithm} 
\author{Renbo Zhao\thanks{Department of Business Analytics, Tippie College of Business, University of Iowa, {renbo-zhao@uiowa.edu}}}
\input{preamble3}
\usepackage[colorlinks=true,breaklinks=true,bookmarks=true,urlcolor=blue,
     citecolor=blue,linkcolor=blue,bookmarksopen=false,draft=false]{hyperref}

\numberwithin{equation}{section}
\numberwithin{assump}{section}
\numberwithin{theorem}{section}
\numberwithin{corollary}{section}
\numberwithin{lemma}{section} 
\numberwithin{example}{section}
\numberwithin{remark}{section}
\numberwithin{definition}{section}
\numberwithin{prop}{section}

\newcommand{\Log}{{\rm Log}}
\newcommand{\Etp}{{\rm Etp}}

\begin{document}

\maketitle

\begin{abstract}

{We present new analysis and algorithm of the dual-averaging-type (DA-type) methods for solving the composite convex optimization problem ${\min}_{x\in\bbR^n} \, f(\rvA x) + h(x)$, where $f$ is a convex and globally Lipschitz function,  $\rvA$ is a linear operator, and $h$ is a ``simple'' and convex function that is used as the prox-function in the DA-type methods. We open new avenues of analyzing and developing DA-type methods, by going beyond the canonical setting where the prox-function $h$ is assumed to be strongly convex (on its domain). To that end, we identify two new sets of assumptions on $h$ (and also $f$ and $\rvA$) and show that they hold broadly for many important classes of non-strongly-convex functions. 
Under the first set of assumptions, we show that  the original DA method still has a $O(1/k)$ primal-dual convergence rate. Moreover, we analyze the affine invariance of this method and its convergence rate. Under the second set of assumptions, we develop a new DA-type method with dual monotonicity, and show that it has a $O(1/k)$ primal-dual convergence rate. Finally, we consider the case where $f$ is only convex and Lipschitz on $\calC:=\rvA(\dom h)$, and construct its globally convex and Lipschitz extension based on the Pasch-Hausdorff envelope. Furthermore, we  characterize the sub-differential and Fenchel conjugate of this extension using the convex analytic objects associated with $f$ and $\calC$.
} 

\end{abstract}

\section{Introduction}\label{sec:intro}

Dual averaging (DA)~\cite{Nest_09} is a fundamental algorithm for solving convex nonsmooth optimization problems. 
In this work we are interested in analyzing DA for the following optimization problem:
\begin{equation}
P_*:= {\min}_{x\in\bbX} \;\{P(x):= f(\rvA x) + h(x)\}.\tag*{$(\rmP)$} \label{eq:P}
\end{equation}
In~\ref{eq:P}, $\rvA:\bbX\to\bbY^*$ is a linear operator, where $\bbX := (\bbR^n,\normt{\cdot}_\bbX)$ and $\bbY:= (\bbR^m,\normt{\cdot}_\bbY)$ are 
normed spaces with dual spaces denoted by $\bbX^* := (\bbR^n,\normt{\cdot}_{\bbX,*})$ and $\bbY^*:= (\bbR^m,\normt{\cdot}_{\bbY,*})$, respectively. 
(Throughout this work, we will simply use $\normt{\cdot}$ to denote the norms on $\bbX$ and $\bbY$, and $\normt{\cdot}_*$ to denote the norms on $\bbX^*$ and $\bbY^*$, when no ambiguity arises.) 
In addition, 
\begin{itemize}[leftmargin=*,itemsep=5pt,parsep=0pt,topsep=1pt]
\item $f:\bbY^*\to \bbR$ is a convex and globally $L$-Lipschitz function, namely
\begin{equation}
\abst{f(z)-f(z')}\le L\normt{z-z'}_*, \quad \forall\, z,z'\in \bbY^*.  \label{eq:global_Lips_f}
\end{equation}
We shall assume that a subgradient of $f$ can be easily computed at any point $z\in\bbY^*$, but {\em do not} assume that the structure of $f$ is so ``simple'' such that the proximal 
operator associated with 
$f$, namely $\prox_{f}(z):= \argmin_{z'\in\bbY^*}\; f(z') + (1/2)\normt{z-z'}_2^2$ for $z\in \bbY^*$, can be easily computed. 
\item $h:\bbX\to \barbbR$ (where $\barbbR:= \bbR\cup\{+\infty\}$) is a proper, closed  and convex function that is ``simple'', in the sense that 
for any $u\in \bbX^*$, we can efficiently find an optimal solution (whenever it exists) to the following problem
\begin{equation}
{\min}_{x\in\bbX} \; \ipt{u}{x} + h(x). \label{eq:subprob} 
\end{equation}
Similar to $f$, we do not assume that $\prox_{h}(x)$ 
can be easily computed for some $x\in\bbX$.  (In fact, the optimization problem involved in 
$\prox_{h}(x)$ is harder to solve than the one 
in~\eqref{eq:subprob} in general.\footnote{
Let $\bbX := (\bbR^n,\normt{\cdot}_2)$ be a Euclidean space and $h^*$ be the Fenchel conjugate of $h$. 
Viewed from the perspective of $h^*$, 
solving the problem in~\eqref{eq:subprob} amounts to computing a subgradient of $h^*$ at $-u\in\bbX$, whereas computing $\prox_{h}(-u)$ amounts to computing $\prox_{h^*}(-u)$ (by Moreau decomposition theorem).})
\end{itemize}
In addition, to make~\ref{eq:P} well-posed, we assume that $P_*> -\infty$.  However, 
we do not need to assume that~\ref{eq:P} has  an optimal solution. 

When $h$ is {\em strongly convex} on its domain, denoted by $\dom h:=\{x\in\bbX:h(x)<+\infty\}$, it effectively acts as a ``prox-function''~\cite{Nest_09}.  Some typical examples of $h$ include i) $h(x)= (1/2)\normt{x}_2^2 + \iota_\calC(x)$, where $\calC\ne \emptyset$ is a closed convex set and $\iota_\calC$ denotes its indicator function, and ii) $h(x)= \sum_{i=1}^n x_i\ln x_i - x_i  + \iota_{\Delta_n}(x)$, where $\Delta_n:=\{x\in\bbR^n:\sum_{i=1}^n x_i=1, \,x\ge 0\}$ denotes the unit simplex. Note that in both examples, the sets $\calC$ and $\Delta_n$ are effectively  {the constraint sets of}~\ref{eq:P}. Indeed, in general, any (closed and convex) constraint set of~\ref{eq:P} can be incorporated into $h$ via its indicator function, and as a result, 
$\dom h$ becomes  the (effective) {\em feasible region of}~\ref{eq:P}. 

With the strongly convex prox-function $h$, the DA method for solving~\ref{eq:P} is shown in Algorithm~\ref{algo:DA}.
Throughout this work, we shall choose the two step-size sequences as follows: 
\begin{equation}
\alpha_k = k+1, \quad \beta_k = k(k+1)/2, \quad \forall\,k\ge 0. \tag*{(Step)} \label{eq:step}
\end{equation}
Based on the above choices, in two seminal works, Grigas~\cite[Section~3.3.1]{Grigas_16}  showed that the DA method converges with rate $O(1/k)$, 
and Bach~\cite[Section~3]{Bach_15} analyzed a version of the mirror descent (MD) method for solving~\ref{eq:P}, and obtained similar computational guarantees as those in \cite[Section~3.3.1]{Grigas_16}. In fact, this comes with no coincidence --- by properly choosing the subgradient of $h$ in the definition of the Bregman divergence induced by $h$, one can indeed establish the equivalence between the MD method in~\cite{Bach_15} and the DA method in Algorithm~\ref{algo:DA} (for details, see~\cite[Section~3.4]{Bach_15} and~\cite[Section~4.1]{Grigas_15}).

The strong convexity of the prox-function plays a critical role in analyzing the DA method for convex nonsmooth optimization problems (see e.g.,~\cite{Nest_09,Grigas_16}). In fact, in 
the literature on the DA method (and its variants), strong convexity has become an integrated component in the definition of the prox-function. 
(We review two exceptions in Section~\ref{sec:related_work}.) 
At the same time, the requirement of strong convexity greatly limits the class of prox-functions that one can work with. In fact, there are many simple functions (for which~\eqref{eq:subprob} can be efficiently solved), which naturally arise in various applications, are not strongly convex on their domains, e.g.,  the log-function ${\rm Log}(x):= -\sum_{i=1}^n \ln x_i$ and the 
entropy function ${\rm Etp}(x):= \sum_{i=1}^n x_i\ln x_i - x_i$ (and also ${\rm Log}+\iota_\calC$ and ${\rm Etp}+\iota_\calC$, where $\calC$ is closed and convex but not bounded). 
This naturally leads  to the following intriguing question: 

\vspace{-2.5ex}
\begin{center}
\mbox{\em Does DA enjoy ``good'' convergence rate even if $h$ is non-strongly-convex on its domain?}
\end{center}
\vspace{-1.5ex}

In the first part of this work, we provide an affirmative answer to this question, and show that under relatively mild assumptions on $h$ (and also $f$ and $\rvA$), the DA method in Algorithm~\ref{algo:DA} indeed has similar computational guarantees to the ``canonical'' case, where $h$ is strongly convex. 
To quickly gain a concrete feeling of our results, let us consider the following simple instance of the problem in~\ref{eq:P}:
\begin{equation}
{\min}_{x\in\bbR^n} \; \big\{P(x):= {\max}_{j\in[m]}\, \ipt{A_j}{x}  - \textstyle \sum_{i=1}^n b_i\ln x_i + {\rm Etp}(b) \big\}, 
\label{P-toy}
\end{equation} 
where the data matrix  $A$ is (entry-wise) {\em  positive}, i.e., $A_{ji}>0$ 
for $j\in[m]$ and $i\in[n]$, 
$A_j\in \bbR^n$ denotes the $j$-th row of $A$ for $j\in[m]$ and $b_i\ge 1$ for $i\in[n]$. 
Let $a_i\in\bbR^m$ denote the $i$-th column of $A$ for $i\in[n]$. In fact, the problem~\eqref{P-toy} 
arises as the dual problem of the following problem:
\begin{equation}
-{\min}_{y\in\bbR^m}\; \big\{D(y) := \textstyle-\sum_{i=1}^n b_i\ln(a_i^\top y)+\iota_{\Delta_m}(y)\big\}, \label{D-toy}
\end{equation}
which appears in some instances of the positron emission tomography problem~\cite{BenTal_01}. Our results suggest that the DA method, 
when applied to the problem~\eqref{P-toy} with the simple parameter choices in~\ref{eq:step}, has the following primal-dual convergence rate guarantee. Specifically, define $\barx^k:= (1/\beta_k)\sum_{i=0}^{k-1} \alpha_i x_i$  and $\bars^k:= s^k/\beta_k $ for $k\ge 1$, and we have 
\begin{align}
P(\barx^k) - D(\bars^k)\le \frac{8\max_{j,j'\in[m]}\, \normt{A_j - A_{j'}}_2^2}{\mu (k+1)}, \quad \forall\, k\ge 1,  \label{eq:conv_toy}
\end{align}
for some constant $\mu >0$. At first glance, such a result may seem somewhat surprising, for two reasons. First, the (nonsmooth) objective function in~\eqref{P-toy} is not strongly convex. According to the classical complexity results of convex optimization (see e.g.,~Nemirovski and Yudin~\cite{Nemi_79}), for convex problems without strong convexity and additional structural assumptions 
(e.g., the proximal operator of $f$ is easily computable), one would expect a $O(1/\sqrt{k})$ convergence rate for a first-order method (and especially subgradient method) that solves it. Second, the log-like function $h(x):= -\sum_{i=1}^n b_i\ln x_i$ is used as the prox-function in the DA method, but it is not strongly convex (on its domain), and one would wonder whether the DA method is even well-defined (in Step~\ref{step:proj}) or converges at all, let alone the $O(1/{k})$ convergence rate. So what are the reasons for the ``nice'' convergence rate in~\eqref{eq:conv_toy}? In fact, as we will see later, the mystery of such a result 
can be precisely explained by two reasons: 
i) all of the primal iterates $\{x^k\}_{k\ge 0}$ produced by the DA method {\em automatically} lie in some convex compact set $\bar\calS\subseteq\dom h$, 
and ii) $h$ is strongly convex on $\bar\calS$.  

\begin{algorithm}[t!]
\caption{Dual Averaging for Solving~\ref{eq:P}}\label{algo:DA}
\begin{algorithmic}
\State {\bf Input}: Pre-starting point $x^{-1}\in\bbX$, step-size sequences $\{\alpha_k\}_{k\ge 0}$ and $\{\beta_k\}_{k\ge 0}$ chosen as in~\ref{eq:step} \vspace{-2ex}
\State {\bf Pre-start}: Compute $g^{-1}\in\partial f(\rvA x^{-1})$, $x^0=\argmin_{x\in\bbX}\; \ipt{g^{-1}}{\rvA x} + h(x)$, $s^0=0$ \vspace{1ex} 
\State {\bf At iteration $k\ge 0$}:
\begin{enumerate}[leftmargin = 6ex]
\item Compute   $g^k \in\partial f(\rvA x^k)$
\item $s^{k+1}:= s^k + \alpha_k g^k$
\item \label{step:proj} $x^{k+1}:=  \argmin_{x\in \bbX}\; 
\ipt{s^{k+1}}{\rvA x}+\beta_{k+1} h(x)$ 
\end{enumerate}
\end{algorithmic}
\end{algorithm}

Of course, the two facts above are not specific to the problem in~\eqref{P-toy}. By judiciously exploiting the structure of the dual problem of~\ref{eq:P},  our analysis reveals that they can happen 
fairly generally for the DA method in solving~\ref{eq:P}.  
Specifically, the first fact  holds as long as the domains of $f^*$ and $h^*$ 
 satisfy a certain {\em inclusion condition} (see Assumption~\ref{assum:Q} for details),  
 where 
$f^*$ and $h^*$ denote the Fenchel conjugates of $f$ and $h$, respectively. 
In addition, there exist several broad families of simple convex functions $h$, including $\Log(\cdot)$ and $\Etp(\cdot)$, 
that are strongly convex on any (nonempty) {\em convex compact} set in its domain --- 
the details are elaborated in Section~\ref{sec:assump}. 

The problem class described above extends the scope of the DA method beyond the strongly convex prox-function $h$, however, there are some problems that do not fall within this class. As a simple example, we still consider the problem in~\eqref{P-toy}, but with  (entry-wise) {\em non-negative} data matrix $A$, i.e., $A_{ji}\ge 0$ for $j\in[m]$ and $i\in[n]$, and $A_j\ne 0$ for $j\in[m]$. (In this case, the dual problem~\eqref{D-toy} arises in applications such as optimal expected log-investment~\cite{Cover_84} and  
learning of multivariate-Hawkes processes~\cite{Zhou_13}.)  In fact, one can easily see that in this situation, the DA method in Algorithm~\ref{algo:DA} is not even well-defined --- specifically, the minimization problem in Step~\ref{step:proj} may not even have an optimal solution  (see Remark~\ref{rmk:well_defined} for details).  

This challenging problem motivates the second part of this work, wherein  we identify 
another new problem class that includes the problem~\eqref{P-toy} with {non-negative} data matrix $A$,    and develop a new DA-type algorithm 
to solve it. 
This new algorithm has a  similar structure to Algorithm~\ref{algo:DA} and employs the same step-size sequences in~\ref{eq:step}, but as a key difference, it generates  dual iterates that keep or improve the dual objective value. We conduct a 
geometric analysis of this algorithm, and show that to obtain an $\varepsilon$-primal-dual gap, the number of iterations needed is of order $O(1/\varepsilon)$.


\setlength\parindent{0pt}

\noindent
\subsection{Main Contributions}

At a high level, our contributions can be categorized into in three main aspects. 
\begin{itemize}[leftmargin = 0ex,label={},topsep=0pt,itemsep=5pt]
\item First, we {\em identify two new problem classes of}~\ref{eq:P} that subsume  and go beyond the ``canonical'' setting where  the prox-function $h$ is strong convex.  
We show that the first problem class can still be solved by the original DA method in Algorithm~\ref{algo:DA}, and {\em develop a new DA-type algorithm} 
for solving the second problem class. Our new models on the prox-function $h$ may also be useful in extending the scope of other first-order methods that involve prox-functions. 
\item Second, we {\em conduct convergence rate analyses} for both the original DA method and the new DA-type algorithm, which 
tackle the two aforementioned problem classes, respectively. 
We show that both methods 
converge  at rate $O(1/k)$ in terms of the primal-dual gap. 

\item 
Third, we {\em develop convex analytic results} that provide 
certificates for important classes of 
convex functions  to satisfy our assumptions on $h$, 
which in turn demonstrates the relatively broad scope of our new models on $h$. These results are algorithm-independent and may be of independent interest. 
\end{itemize}

\vspace{1ex}

At a more detailed level, our main contributions are summarized as follows.
\begin{enumerate}[label=(\arabic*),leftmargin=16pt,itemsep=5pt,parsep=0pt,topsep=1pt]
\item  \label{item:DA_contri}
We show that the original  DA method in Algorithm~\ref{algo:DA} has a primal-dual convergence rate of $O(1/k)$ 
when applied to solving~\ref{eq:P}, under two assumptions: (i) the prox-function $h$ is strongly convex on any  nonempty convex compact set inside 
$\dom h$, and (ii) $-\rvA^*(\cl \dom f^*)\subseteq \inter \dom h^*$ (where $\cl$ and $\inter$ stands for closure and interior, respectively), both of which are {\em strictly weaker} than the strong convexity assumption of $h$ (cf.~Lemma~\ref{lem:relax}). In addition, 
we show that the first assumption above (on $h$)  is satisfied broadly by non-strongly-convex functions ---  in particular, it is satisfied by any separable, very strictly convex and Legendre function 
whose Hessian ``blows up'' on the boundary of its domain (cf.~Lemma~\ref{lem:sep}).

\item We develop a new DA-type method in Algorithm~\ref{algo:MDA} for solving~\ref{eq:P} under two assumptions. The first one is the same as assumption (i) in Point~\ref{item:DA_contri} above, and the second one assumes that $\dom h^*$ is open. This new method has a simple structure similar to the original DA method in Algorithm~\ref{algo:DA}, but as a key difference,  it generates  dual iterates that keep or improve the dual objective value. We show that to obtain an $\varepsilon$-primal-dual gap, the number of iterations needed by this new method is of order $O(1/\varepsilon)$ (cf.\ Remark~\ref{rmk:interpret_algo2}). In addition, based on the notion of {\em affine attainment}, we provide certificates to identify important non-strongly-convex functions $h$ such that $\dom h^*$ is open (cf.\ Section~\ref{sec:assump_dom_open}).

\item We provide a detailed discussion on the affine invariance of Algorithm~\ref{algo:DA}  and its convergence rate  analysis. Specifically, we first show that under the bijective affine re-parameterization, the re-parameterized problem still satisfies the two assumptions in Point~\ref{item:DA_contri}, and hence Algorithm~\ref{algo:DA} is well-defined on this problem. We then show that i) Algorithm~\ref{algo:DA} is affine-invariant, and ii) if the norm $\normt{\cdot}_\bbX$ is induced by some set that is intrinsic to~\ref{eq:P}, then the convergence rate analysis of Algorithm~\ref{algo:DA} is also affine-invariant (cf.\ Section~\ref{sec:aff_inv_analysis}).

\item We relax the globally convex and Lipschitz assumptions of $f$, by only assuming that $f$ is convex and $L$-Lipschitz on $\calC:=\rvA(\dom h)$. 
Indeed, by leveraging the notion of Pasch-Hausdorff envelope (see~\cite[Section~12]{Bauschke_11}), we can obtain a globally convex and $L$-Lipschitz extension of $f$, denoted by $F_L$. 
We characterize 
$\partial F_L(z)$ for $z\in \calC$, as well as $F_L^*$ 
and $\dom F_L^*$, in terms of convex analytic objects associated with $f$ and $\calC$. 
These results go beyond the scope of Algorithms~\ref{algo:DA} and~\ref{algo:MDA}, and apply to any ({feasible}) first-order method that requires $f$ 
to be globally convex and Lipschitz. 
 
\end{enumerate}

\noindent
\subsection{Notations}
Let $\bbU:=(\bbR^d, \normt{\cdot})$ be a normed space. 
For a nonempty set $\calU\subseteq \bbU$, we denote its interior, relative interior, boundary, closure, affine hull, convex hull, conic hull and complement by $\inter\calU$, $\ri\calU$, $\bdry\calU$, $\cl \calU$,  $\aff \calU$, $\conv \calU$, $\cone \calU$ and $\calU^c$, respectively. We call $\calU$ solid if $\inter\calU\ne\emptyset$. 
Given two nonempty sets $\calA,\calB\subseteq\bbU$, define 
\begin{equation}
\dist_{\normt{\cdot}}(\calA, \calB):= {\inf} \{\normt{u-u'}: u\in\calA,\; u'\in \calB \}, \label{eq:def_dist}
\end{equation}
and for any $u\in\bbU$, define $$\dist_{\normt{\cdot}}(u, \calB):= {\inf} \{\normt{u-u'}: u'\in\calB \}.$$ 
Given an affine subspace $\calA\subseteq\bbU$, denote the linear subspace associated with $\calA$ by $\lin \calA$, namely $\lin \calA:= \calA - u_0$, for any $u_0\in\calA$. 
Given a linear operator $\rvT:\bbU\to \bbU^*$, denote its adjoint by $\rvT^*:\bbU\to \bbU^*$, namely $\ipt{\rvT u}{u'} = \ipt{\rvT^* u'}{ u}$ for all $u,u'\in\bbU$, and define its operator norm $\normt{\rvT}:= \max_{\normt{u}=1}\; \normt{\rvT u}_*$. If $\rvT$ is self-adjoint (i.e., $\rvT = \rvT^*$), define its minimum eigenvalue $$\lambda_{\min}(\rvT):= {\min}_{\normt{u}=1}\, \ipt{\rvT u}{u}\in\bbR,$$ and  we call $\rvT$ positive definite (denoted by $\rvT\succ 0$) if $\lambda_{\min}(\rvT)>0$. 
With slight overload of notation, we denote a real symmetric and positive definite matrix $X$ by $X\succ 0$. 
For a proper closed convex function $\psi:\bbU\to \barbbR$,  let $\psi^*:\bbU^* \to \barbbR $ denote its Fenchel conjugate, namely
$$\psi^*(w)= {\sup}_{u\in\bbU}\; \ipt{w}{u} - \psi(u).$$ 
In addition, define $\bbR_{++}:=(0,+\infty)$,  $\bbR_+:=[0,+\infty)$ and $\bbR_{--}:=(-\infty,0)$, and  let $e_j\in \bbR^d$ denote the $j$-th standard coordinate vector    
(i.e., the $j$-th column of the identity matrix $I_d$) and $e:= \sum_{j=1}^d e_j$. 
Also, define $$\Delta_d:=\{x\in\bbR^d:e^\top x=1, \,x\ge 0\}.$$

\section{Assumptions and Their Implications} \label{sec:assump}


\noindent 
We introduce the following two assumptions, 
either of which is strictly weaker than the strong convexity assumption of $h$ on its domain (see Lemma~\ref{lem:relax} below). 
 
\begin{assump} \label{assum:h}
For any nonempty convex compact set $\calS \subseteq \dom h$, there exists $\mu_\calS > 0$ such that $h$ is $\mu_\calS$-strongly-convex  on $\calS$ w.r.t.\ $\normt{\cdot}_{\bbX}$, i.e., for all $x,x'\in \calS$ and all $\lambda\in(0,1)$, 
\begin{equation}
h(\lambda x + (1-\lambda) x')\le \lambda h(x) + (1-\lambda) h(x') - \frac{\lambda(1-\lambda)\mu_\calS}{2}\normt{x-x'}_\bbX^2 . 
\end{equation}
For singleton $\calS$, we let $\mu_\calS:=1$. 
  For non-singleton $\calS$, we let $\mu_\calS$ take the tightest value, i.e.,  
\begin{equation}
\mu_\calS: = \inf\left\{\frac{\lambda h( x) + (1-\lambda)h(x')-h((1-\lambda)x'+\lambda x)  }{(\lambda(1-\lambda)/2)\normt{x'-x}_\bbX^2 }: \; x, x'\in\calS,\, x\ne x',\, \lambda\in (0,1)\right\}. 
\label{eq:def_muS}
\end{equation}
\end{assump}

\begin{remark}[Effects of $\normt{\cdot}_\bbX$ on $\mu_\calS$] \label{rmk:choice_norm}
Since all norms are equivalent on finite-dimensional normed spaces, the choice of $\normt{\cdot}_\bbX$ only affects the value of $\mu_\calS$, but not its positivity. In other words, if $h$ satisfies Assumption~\ref{assum:h} under a particular norm on $\bbX$, then it satisfies Assumption~\ref{assum:h} under any other norm on $\bbX$. That said, in Section~\ref{sec:aff_inv_analysis}, we will see that to make $\mu_\calS$ invariant to certain affine re-parameterization of~\ref{eq:P}, it is important that we choose $\normt{\cdot}_\bbX$ in an appropriate way. 
\end{remark}

Assumption~\ref{assum:h} has the following important implications about $h^*:\bbX^*\to\barbbR$, namely the Fenchel conjugate of $h$. 

\begin{lemma} \label{lem:h^*_diff}
Under Assumption~\ref{assum:h}, $\inter\dom h^*\ne \emptyset$ and 
$h^*$ is continuously differentiable on $\inter\dom h^*$. 
In addition, if $\dom h$ is non-singleton, then $h$ is strictly convex on $\dom h$. 
\end{lemma}

\begin{proof}
The first part of the lemma trivially holds if $\dom h$ is a singleton, in which case $h^*$ is an affine function. Thus we focus on non-singleton $\dom h$, and we first show that in this case, $h$ is strictly convex on $\dom h$. Indeed, for any $x,y\in \dom h$, $x\ne y$,  since $[x,y]$ is convex and compact, by Assumption~\ref{assum:h}, there exists $\mu>0$ such that for all $\lambda \in(0,1)$, 
\begin{equation}
h(\lambda x + (1-\lambda) y)\le \lambda h(x) + (1-\lambda) h(y) - \frac{\lambda(1-\lambda)\mu}{2}\normt{x-y}^2 < \lambda h(x) + (1-\lambda) h(y). \label{eq:strict_conv_h}
\end{equation}
Next, we show that $\inter\dom h^*\ne \emptyset$ by contradiction. Suppose that $\inter\dom h^*= \emptyset$, then its affine hull $\calA  \subsetneqq\bbX^*$. 
Define $\calL:=\lin\calA$, then there exists $ d\in \bbX$ such that $d\ne 0$ and $d\perp \calL$, i.e., $\ipt{d}{u} = 0$ for all $u\in\calL$. Now, fix any $u_0\in\calA$. Since $h$ is closed and convex, for any $x\in \dom h$ and  $\lambda\in\bbR$, 
\begin{align}
h(x+\lambda d) &= {\sup}_{u\in\calA}\; \ipt{x+\lambda d}{u} - h^*(u)\\
& =  {\sup}_{u\in\calA}\; \ipt{x}{u} - h^*(u) + \lambda \ipt{d}{u_0} \label{eq:d_u0}\\
& =  h(x) + \lambda \ipt{d}{u_0},
\end{align}
where~\eqref{eq:d_u0} follows from that $\calA = u_0 + \calL$ and  $d\perp \calL$. 
This implies that for $\lambda\in (0,1)$,
\begin{equation}
h(x+\lambda d)  = \lambda (h(x) + \ipt{d}{u_0}) + (1-\lambda) h(x) = \lambda  h(x+d) + (1-\lambda) h(x),
\end{equation}
contradicting~\eqref{eq:strict_conv_h}.  
Lastly, we show that $h^*$ is continuously differentiable on $\inter\dom h^*$. 
For any $u\in \inter\dom h^*$, from~\cite[Fact 2.11]{Bauschke_97}, we know that $h - \ipt{u}{\cdot}$ is coercive, and with the strict convexity of $h$, we know that $\argmax_{x\in\bbX} \; \ipt{u}{x} - h(x)$ exists and is unique, and hence $h^*$ is differentiable at $u$. In addition, since $h^*$ is proper, closed and convex,   
by~\cite[Theorem 25.5]{Rock_70}, we know that $\nabla h^*$ is continuous on $\inter\dom h^*$. This completes the proof. 
\end{proof}



%
%
\noindent 
Before stating our second assumption, for notational convenience, let us define
\begin{equation}
\calQ:=\cl \dom f^*\quad \andd\quad \calU:=-\rvA^*(\calQ).  \label{eq:def_QU}
\end{equation}

\vspace{.5ex}

\begin{assump} \label{assum:Q}
$\calU\subseteq \inter \dom h^*$. 
\end{assump}

\noindent 
\begin{remark}[Verifying Assumption~\ref{assum:Q}] \label{rmk:verify}
Two remarks are in order. First, in many applications, the convex functions $f$ and $h$ have  relatively simple analytic structures, which enable us to find 
(the domains of) their Fenchel conjugates $f^*$ and $h^*$ relatively easily. 
Take the problem in~\eqref{P-toy} with (entry-wise) positive data matrix $A$ as an example. In this case, since $f(z) = {\max}_{j\in[m]}\, z_j$, $\rvA: x\mapsto Ax$ and $h(x) = - \textstyle \sum_{i=1}^n b_i\ln x_i$, we clearly have $f^*(y):= \iota_{\Delta_m}(y)$ and $h^*(u):= -\sum_{i=1}^n b_i\ln(-u_i)$, and hence 
$\calQ:= \Delta_m$ and $\calU = -\conv\{A_i\}_{i=1}^m\subseteq -\bbR_{++}^n =\dom h^* = \inter \dom h^*$, which verifies Assumption~\ref{assum:Q}. For examples of the Fenchel conjugates of many important convex functions and their calculus rules, we refer readers to~\cite[Chapter~4]{Beck_17b}. 
Second, in some scenarios, we do not need to know $h^*$ in order to find  $\dom h^*$, and in fact, 
finding $\dom h^*$ sometimes can be much easier compared to finding $h^*$ (and the same applies to $f^*$). 
As an important case, consider $h = h_1 + h_2$ 
for some ``simple'' convex functions $h_1$ and $h_2$ such that  $h_1^*$ and $h_2^*$ can be easily found (in closed forms). For example, we can let $h_1:x\mapsto -\sum_{i=1}^n \ln x_i$ and $h_2 := \iota_{\calC}$, where  
$\calC$ is a closed convex set  that satisfies $\calC\cap\bbR_{++}^n\ne \emptyset$.    If $\ri\dom h_1\cap \ri\dom h_2 \ne\emptyset$, then 
we know that $h^* = h_1^* \,\square\, h_2^*$, i.e., the infimal convolution of $h_1$ and $h_2$.  
Note that even if both $h_1^*$ and $h_2^*$ have simple structures, their infimal convolution can often be difficult to compute. 
In contrast, $\dom h^*$ can be easily obtained 
in this case, namely 
$\dom h^* = \dom h_1^* + \dom h_2^*$. For more details and examples, we refer readers to Proposition~\ref{prop:F_L*} and Remark~\ref{rmk:F_L*}. 
\end{remark}

\vspace{.5ex}

\begin{remark}[Well-Definedness of Algorithm~\ref{algo:DA}] \label{rmk:well_defined}
Since $h$ may not be strongly convex on its domain, the minimization problem in Step~\ref{step:proj} of Algorithm~\ref{algo:DA} may not have an optimal solution, making Algorithm~\ref{algo:DA} ill-defined.
It turns out that, on top of Assumption~\ref{assum:h},  if Assumption~\ref{assum:Q} holds, then  Algorithm~\ref{algo:DA} is well-defined for any pre-starting point $x^{-1}\in\bbX$ (see Lemma~\ref{lem:well_posed} in Section~\ref{sec:DA_analysis}). 
On the other hand, if Assumption~\ref{assum:Q} fails, then there exist problem instances on which Algorithm~\ref{algo:DA} is ill-defined. To see this, take the problem in~\eqref{P-toy} as an example, wherein 
the data matrix $A\in\bbR_+^{m\times n}$ satisfy that $a_1 = e_1$ and $A_1=e_1$. 
From Remark~\ref{rmk:verify}, we know that  $\calU = -\conv\{A_i\}_{i=1}^m\not\subseteq -\bbR_{++}^n= \inter \dom h^*$, implying that Assumption~\ref{assum:Q} fails. 
In this case,  if we choose $x^{-1} = e_1$, then $g^{-1} = e_1$ and the problem $\min_{x\in\bbR^n}\; \ipt{A_1}{x} - \textstyle \sum_{i=1}^n b_i\ln x_i$,  whose optimal solution defines $x^0$, has no optimal solution at all. This makes Algorithm~\ref{algo:DA} ill-defined. 
\end{remark}

Assumptions~\ref{assum:h} and~\ref{assum:Q} together  have the following important implications. 

\begin{lemma} \label{lem:compact}
Under Assumptions~\ref{assum:h} and~\ref{assum:Q}, define 
\begin{equation}
\barcalS := \conv(\nabla h^*(\calU)).  \label{eq:def_barS}
\end{equation}
Then $\calQ$, $\calU$ and 
$\barcalS $ are all nonempty, convex and compact, and $\barcalS\subseteq \dom h$.  
Furthermore, 
$h$ is $\mu_{\barcalS}$-strongly-convex on $\barcalS$, where $\mu_{\barcalS}>0$ is defined in~\eqref{eq:def_muS}. 
\end{lemma}

\begin{proof}
Since $f$ is convex and globally Lipschitz, $\dom f^*$ is nonempty, convex and bounded (cf.~\cite[Corollary~13.3.3]{Rock_70}). Thus $\calQ=\cl \dom f^*$ is nonempty, convex and compact, and so is $\calU$. 
Since 
$\calU\subseteq \inter \dom h^*$ and $\nabla h^*$ is continuous on $\inter \dom h^*$ (cf.~Lemma~\ref{lem:h^*_diff}), 
we know that $\nabla h^*(\calU)\ne \emptyset$ is compact. As a result, $\calS\ne \emptyset$ is convex and compact. 
Since $\ran \nabla h^* \subseteq \dom h$ (where $\ran \nabla h^*$ denotes the range of $\nabla h^*$), we have $\nabla h^*(\calU)\subseteq \dom h$, and since $\dom h$ is convex, 
we have $\calS\subseteq \dom h$. 
\end{proof}

Lastly, we show that Assumptions~\ref{assum:h} and~\ref{assum:Q} 
are strictly weaker than the strong convexity assumption of $h$. 

\begin{lemma}\label{lem:relax}
If $h$ is strongly convex on its domain, then Assumptions~\ref{assum:h} and~\ref{assum:Q} hold, but the converse may not be true. 
\end{lemma}

\begin{proof}
If $h$ is strongly convex on its domain, then clearly i) Assumption~\ref{assum:h} holds  and ii) $\inter \dom h^* = \dom h^* = \bbX^*$ and Assumption~\ref{assum:Q} holds. To see that the converse may not be true, we can simply use~\eqref{P-toy} as a counterexample. 
\end{proof}



\subsection{Certificates for Assumption~\ref{assum:h}} \label{sec:suff_assump_h}
One sufficient condition to ensure that Assumption~\ref{assum:h} holds is shown in 
the following lemma.

\begin{lemma} \label{lem:suff_h}
Let 
$h$ 
be closed, convex and 
twice continuously differentiable on $\inter\dom h\ne\emptyset$. 
Given a nonempty convex compact set $\calS \subseteq \dom h$, if there exists $z\in \inter\dom h$ such that 
\begin{equation}
\kappa_{\calS_z}:={\inf}_{x\in \calS_z^o}\; \lambda_{\min} (\nabla^2 h(x)) > 0,
\label{eq:kappa_S}
\end{equation}
where 
$\calS_z:= \conv(\calS\cup\{z\})$ and $\calS_z^o:=\calS_z\cap\inter\dom h$,  
then $h$ is $\mu_\calS$-strongly convex on $\calS$ and $\mu_\calS\ge \kappa_{\calS_z}$. 
\end{lemma}

\begin{proof}
See Appendix~\ref{app:proof_suff_h}. 
\end{proof}

\noindent 
By Lemma~\ref{lem:suff_h}, we immediately 
have the following examples of $h$ that satisfy Assumption~\ref{assum:h}.

\begin{example}\label{eg:sep_h}
The following examples of $h$ satisfy Assumption~\ref{assum:h}:
\begin{itemize}
\item $h(x):= \sum_{i=1}^n -\ln x_i$ for $x>0$
\item $h(x):= \sum_{i=1}^n x_i \ln x_i - x_i$ for $x\ge 0$
\item $h(x):= \sum_{i=1}^n \exp(x_i)$ for $x\in\bbR^n$
\item $h(x):= \sum_{i=1}^n x_i^{-p}$ for $x>0$, where $p > 0$
\item $h(X):= -\ln\det (X)$ for $X\succ 0$
\end{itemize}
\end{example}

\noindent
By definition, 
given an instance of $h$ that satisfies Assumption~\ref{assum:h}, we can generate new instances satisfying Assumption~\ref{assum:h} by incorporating indicator functions of some closed convex sets into it. 

\begin{lemma}
If $h$ satisfies Assumption~\ref{assum:h}, then for any closed convex set $\calC$ such that $\calC \cap \dom h\ne \emptyset$, 
the function $h+\iota_\calC$ is proper, closed and convex, and satisfies Assumption~\ref{assum:h} as well.
\end{lemma}

\noindent
{\bf Connections  to the very strictly convex Legendre functions.} 
In fact, all of the examples in Example~\ref{eg:sep_h} fall under the class of ({separable}) very strictly convex  Legendre functions, which  was introduced in~\cite{Bauschke_00}. Let us provide the formal definitions of this class of functions below.  

\begin{definition}[{Legendre  function;~\cite[Definition~2.1]{Bauschke_00}}]
Let $h$ be 
a closed and convex function with $\inter\dom h\ne \emptyset$. We call $h$ {Legendre} if it is 
 i) {\em essentially smooth}, namely   it is (continuously) differentiable on $\inter\dom h$, and furthermore, if $\bdry\dom h\ne \emptyset$, then  for any $\{x^k\}_{k\ge 0}\subseteq \inter\dom h$ such that $x^k\to x\in \bdry\dom h$, we have $\normt{\nabla h(x^k)}_*\to +\infty$,
  and ii)  {\em essentially strictly convex}, namely it is strictly convex on $\inter\dom h$.  
\end{definition}

\noindent
The class of Legendre functions enjoy the following  nice properties. 

\begin{lemma}[{\cite[Theorem~26.5]{Rock_70}}]\label{lem:legendre}
If $h$ is Legendre, so is $h^*$, and $\nabla h:\inter\dom h\to \inter\dom h^*$ is a homeomorphism, whose inverse $(\nabla h)^{-1} = \nabla h^*$. 
\end{lemma}

\begin{definition}[{Very strictly convex function;~\cite[Definition~2.8]{Bauschke_00}}] \label{def:vsc}
Let $h$ be proper, closed and convex with $\inter\dom h\ne \emptyset$. We call $h$ {very strictly convex}
if it is twice continuously differentiable on $\inter\dom h$ and $\nabla^2 h(x)\succ 0$ for all $x\in\inter\dom h$.
\end{definition}

\begin{remark}[Strict convexity, very strict convexity and Assumption~\ref{assum:h}]  \label{rmk:sc_vsc}
Note that a strictly convex function may not be very strict convex or satisfy Assumption~\ref{assum:h}. A prototypical counterexample would be $h(x):= x^4$ for $x\in\bbR$, since $h''(0) = 0$. 
On the other hand, if $h$ is very strictly convex, it is clearly strictly convex on $\inter\dom h$, but may not be strictly convex on $\bdry\dom h$. To see this, consider $h(s,t):= s^3/t$ with $\dom h = \bbR_+\times \bbR_{++}$ and $h(0,0):= 0$. 
Note that in this case, $h$ does not satisfy Assumption~\ref{assum:h} either. Lastly, if $h$ satisfies Assumption~\ref{assum:h}, it may not be twice differentiable on $\inter\dom h$ and hence not very strictly convex. However,  from the  proof of Lemma~\ref{lem:h^*_diff}, if $\dom h$ is non-singleton, then $h$ is indeed strictly convex (on its domain). 
\end{remark}

From Remark~\ref{rmk:sc_vsc}, we know that very strict convexity does not imply Assumption~\ref{assum:h}. That said, in the special case where $\dom h = \bbX$, the implication  is indeed true. 

\begin{lemma} \label{lem:vsc_legendre_full_domain}
If $h$ is very strictly convex, then it is $\mu_\calS$-strongly convex on any nonempty and compact set $\calS\subseteq \inter\dom h$, and 
\begin{equation}
\mu_\calS\ge \eta_{\calS}:={\min}_{x\in\calS}\; \lambda_{\min}(\nabla^2 h(x))>0. 
\end{equation}
In particular, if $\dom h = \bbX$, then $h$ satisfies Assumption~\ref{assum:h}. 
\end{lemma}

\begin{proof}
The first part of the lemma follows from Definition~\ref{def:vsc} and the compactness of $\calS$. The second part of the lemma follows from $\dom h=\inter\dom h$ (since $ \dom h = \bbX$). 
\end{proof}

\noindent 
Lemma~\ref{lem:vsc_legendre_full_domain} deals with the case where $\dom h = \bbX$, or equivalently, $\bdry\dom h=\emptyset$. Let us now focus on the case where $\bdry\dom h\ne \emptyset$. 
We call $h$  {\em separable} if $h(x) = \sum_{i=1}^n h_i(x_i)$ for univariate functions $h_i:\bbR\to \barbbR$, $i\in[n]$. The next result shows that if $h$ is a separable, very strictly convex and Legendre function whose Hessian ``blows up'' on the boundary of its domain, then it satisfies Assumption~\ref{assum:h}. 

\begin{lemma} \label{lem:sep}
Let $h$ be separable, very strictly convex and Legendre, 
and 
$\normt{\nabla^2 h(x^k)}\to +\infty$ for any $\{x^k\}_{k\ge 0}\subseteq \inter\dom h$ such that $x^k\to x\in \bdry \dom h\ne \emptyset$. 
Then $h$ satisfies Assumption~\ref{assum:h}. 
\end{lemma}

\begin{proof}
See Appendix~\ref{app:proof_sep}. 
\end{proof}

\noindent 
Note that all of the examples in Example~\ref{eg:sep_h} satisfy the conditions in Lemma~\ref{lem:sep}, which provides another way for us to see that these examples satisfy Assumption~\ref{assum:h}. 
In general, it is unclear whether the entire class of very strictly convex Legendre functions satisfy Assumption~\ref{assum:h} (and we leave this to future investigation). 
Nevertheless, as long as $h$ is very strictly convex  and Legendre,  
the ``essential'' results in Lemmas~\ref{lem:h^*_diff} and~\ref{lem:compact} still hold under Assumption~\ref{assum:Q} (see Lemma~\ref{lem:vsc_suff} below),  
and as a result, all of our results in the subsequent sections still hold in this case.  
{\em In view of this, $h$ being very strictly convex  and Legendre can be regarded as an alternative assumption to Assumption~\ref{assum:h}.}



\begin{lemma}\label{lem:vsc_suff}
Under Assumption~\ref{assum:Q}, if $h$ is very strictly convex and Legendre, then $\inter\dom h^*\ne \emptyset$ and  $h^*$ is continuously differentiable on $\inter\dom h^*$. In addition, $\barcalS\subseteq \inter\dom h$ is nonempty, convex and compact, and $h$ is $\mu_{\barcalS}$-strongly-convex on $\barcalS$. 
\end{lemma}

\begin{proof}
Since $h$ is Legendre, by Lemma~\ref{lem:legendre}, we know that $h^*$ is Legendre, and by definition, $\inter\dom h^*\ne \emptyset$ and $h^*$ is continuously differentiable on $\inter\dom h^*$. 
Since $\calU\subseteq \inter \dom h^*$ and $\nabla h^*$ is continuous on $\inter \dom h^*$,
we know that $\nabla h^*(\calU)\ne \emptyset$ is compact. As a result, $\barcalS$ is nonempty, convex and compact. Since $\ran\nabla h^*=\inter\dom h$, we have $\nabla h^*(\calU)\subseteq \inter\dom h$, and since  $\inter\dom h$ is convex, we have $\barcalS\subseteq \inter\dom h$. 
Since $h$ is very strictly convex, by Lemma~\ref{lem:vsc_legendre_full_domain}, we know that $h$ is $\mu_{\barcalS}$-strongly-convex on $\barcalS$. 
%
\end{proof}

\section{Convergence Rate Analysis of  Algorithm~\ref{algo:DA} } \label{sec:DA_analysis}

For ease of exposition, in this section, 
we will base our analysis of Algorithm~\ref{algo:DA} on Assumptions~\ref{assum:h} and~\ref{assum:Q}. Readers should keep in mind that by Lemma~\ref{lem:vsc_suff}, our analysis still work with  Assumptions~\ref{assum:h} replaced by $h$ being very strictly convex  and Legendre.

To start, let us define 
\begin{equation}
\bars^0 := g^{-1}\quad \andd \quad \bars^k:= s^k/\beta_k, \quad \forall\, k\ge 1.  \label{eq:def_bars}
\end{equation}

\noindent
Note that Step~\ref{step:proj} in Algorithm~\ref{algo:DA} is {\em well-defined} when $h$ is strongly convex (on its domain), namely, 
the minimization problem therein has a unique optimal solution. 
Of course,  this is not the case in general when $h$ is not strongly convex. However, as suggested by the lemma below, under Assumptions~\ref{assum:h} and~\ref{assum:Q}, Algorithm~\ref{algo:DA} is indeed {well-defined}.

\begin{lemma} \label{lem:well_posed}
Under Assumptions~\ref{assum:h} and~\ref{assum:Q}, in Algorithm~\ref{algo:DA}, 
 $\bars^k\in \calQ$ and $x^{k}= \nabla h^*(-\rvA^*\bars^{k})\in\barcalS$ for $k\ge 0$. 
\end{lemma}

\begin{proof}
Note that by the definition of $\{\bars^k\}_{k\ge 0}$,  
we have 
\begin{equation}
x^{k}:=  {\argmin}_{x\in \bbX}\; \ipt{\bars^{k}}{\rvA x} + h(x), \qquad \forall\, k\ge 0.  \label{eq:def_x_bars}
\end{equation}
Let us prove Lemma~\ref{lem:well_posed} by induction. When $k=0$, $\bars^0=  g^{-1} \in \partial f(\rvA x^{-1}),$ 
which implies that $\bars^0\in \dom f^*\subseteq \calQ$. By Assumption~\ref{assum:Q}, we know that $-\rvA^*\bars^{0}\in\calU\subseteq \inter\dom h^*$. By Lemma~\ref{lem:h^*_diff}, we know that $h^*$ is differentiable at $-\rvA^*\bars^{0}$, and by~\eqref{eq:def_x_bars}, 
we know that $x^{0} = \nabla h^*(-\rvA^*\bars^{0}) \in\barcalS$. Now, suppose that $\bars_k\in \calQ$ and $x^{k}= \nabla h^*(-\rvA^*\bars^{k})\in\barcalS$ for some $k\ge 0$. Note that in~\ref{eq:step}, we have 
\begin{equation}
\beta_{k+1} = \beta_{k} + \alpha_k, \quad \forall\, k\ge 0,  \label{eq:beta_alpha}
\end{equation}
and hence for all $k\ge 0$, 
\begin{equation}
\bars^{k+1} = \frac{s^{k+1}}{\beta_{k+1}} = \frac{s^k + \alpha_k g^k}{\beta_{k} + \alpha_k} = \frac{\beta_{k}  }{\beta_{k} + \alpha_k} \bars^k + \frac{\alpha_k }{\beta_{k} + \alpha_k} g^k. 
\end{equation}
Since $g^k\in\calQ$ and $\bars^k\in\calQ$, and $\calQ$ is convex, we know that $\bars^{k+1} \in\calQ$.
 Repeating the same argument above, 
we know that $x^{k+1} = \nabla h^*(-\rvA^*\bars^{k+1}) \in\barcalS$. 
\end{proof}

Next, we provide primal-dual convergence rate of the DA method in Algorithm~\ref{algo:DA}. To that end, let us write down the (Fenchel-Rockefeller) dual problem of~\ref{eq:P}:  
\begin{equation}
-D_* := -{\min}_{y\in\bbY} \;\,\{D(y):=  h^*( - \rvA^*y) + f^*(y)\}.\tag*{(D)} \label{eq:D}
\end{equation}

\begin{theorem}\label{thm:pdgap}
In Algorithm~\ref{algo:DA}, 
define  
\begin{equation}
\barx^k:= (1/\beta_k)\textstyle\sum_{i=0}^{k-1} \alpha_i x_i \quad \andd\quad  \tilx^k\in \argmin_{x\in\{x_0, \ldots\,x_{k-1}\}} P(x),\quad \forall \,  k\ge 1.  \label{eq:def_barx_tilx}
\end{equation}
Under  Assumptions~\ref{assum:h} and~\ref{assum:Q}, 
for any pre-starting point $x^{-1}\in \bbX$,  we have 
\begin{align}
&\max\{P(\barx^k) + D(\bars^k),  P(\tilx^k)+D(\bars^k)\}\le \frac{8\diam_{\normt{\cdot}_*}(\calU)^2}{\mu_{\barcalS}(k+1)}, \quad \forall \,  k\ge 1, \label{eq:comp_guarantee_DA}  \\
&\qquad \where\quad \diam_{\normt{\cdot}_*}(\calU):= \textstyle\max_{u,u'\in\calU}\, \normt{u-u'}_*=  \max_{y,y'\in\calQ}\, \normt{\rvA^* (y-y')}_*. \label{eq:def_diamU}
\end{align}
\end{theorem}

\begin{proof}
We adopt the convention that empty sum equals zero. 
For $k\ge 0$, define 
\begin{equation}
\psi_k(x):= \textstyle\sum_{i=0}^{k-1} \alpha_i (f(\rvA x_i) + \ipt{g_i}{\rvA (x-x_i)}) + \beta_k h(x) \quad \andd \quad \psi_k^*:= {\min}_{x\in\bbX}\,\psi_k(x).  \label{eq:psi_k} 
\end{equation}
Note that $\psi_0 \equiv 0$ and for $k\ge 0$, we have 
\begin{equation}
x^k \in {\argmin}_{x\in\bbX}\; \psi_k(x) \quad\Longrightarrow \quad  \psi_k^*=\psi_k(x^k). 
\end{equation}
%
In addition, since $h$ is $\mu_{\barcalS}$-strongly convex on $\barcalS$, $\psi_k$ is $(\beta_k\mu_{\barcalS})$-strongly convex on $\barcalS$, for all $k\ge 0$. As a result, 
 for all $x\in\barcalS$, we have 
\begin{align}
\psi_k(x) &\ge \psi_k(x^k) + ({\beta_k\mu_{\barcalS}}/{2})\normt{x-x^k}^2,\label{eq:psik_sc} \\
h(x)&\ge h(x^k) + \ipt{ - \rvA^*\bars^{k}}{ x-x^k} + ({\mu_{\barcalS}}/{2})\normt{x-x^k}^2, \label{eq:h_sc}
\end{align}
where~\eqref{eq:h_sc} follows from $- \rvA^*\bars^{k}\in \partial h(x^k)$ (since $x^{k}= \nabla h^*(-\rvA^*\bars^{k})$ by Lemma~\ref{lem:well_posed}). 
 Now, for all $k\ge 0$ and all $x\in\barcalS$, we have 
\begin{align}
\psi_{k+1}(x) &= \psi_k(x) + \alpha_k (f(\rvA x^k) + \ipt{g^k}{\rvA (x-x^k)}) + (\beta_{k+1} - \beta_k) h(x)\\
&\gea \psi_k(x^k) + ({\beta_k\mu_{\barcalS}}/{2})\normt{x-x^k}^2  + \alpha_k f(\rvA x^k) + \alpha_k\ipt{g^k}{\rvA (x-x^k)} \nn\\
&\hspace{8em} + \alpha_k (h(x^k) - \ipt{\bars^{k}}{\rvA (x-x^k)} + ({\mu_{\barcalS}}/{2})\normt{x-x^k}^2)   \\
&\ge \psi_k(x^k) + \alpha_k P(x^k) + ({\beta_{k+1}\mu_{\barcalS}}/{2})\normt{x-x^k}^2  +  \alpha_k\ipt{g^k-\bars^{k}}{\rvA (x-x^k)}\\
&\geb \psi_k(x^k) + \alpha_k P(x^k)  - \frac{2\alpha_k^2}{{\beta_{k+1}\mu_{\barcalS}}} \normt{\rvA^*(g^k-\bars^{k})}_*^2\\
&\gec \psi_k(x^k) + \alpha_k P(x^k)  - \frac{2\alpha_k^2}{{\beta_{k+1}\mu_{\barcalS}}} \diam_{\normt{\cdot}_*}(\calU)^2, \label{eq:telescope}
\end{align}
where (a) follows from~\eqref{eq:psik_sc},~\eqref{eq:h_sc} and~\ref{eq:step},  
(b) follows from Young's inequality and (c) follows from that both $g^k,\bars^{k}\in\calQ$ for $k\ge 0$ (cf.~Lemma~\ref{lem:well_posed}) and the definition of $\diam_{\normt{\cdot}_*}(\calU)$ in~\eqref{eq:def_diamU}.  Now, choose $x = x^{k+1}\in\barcalS$ and telescope~\eqref{eq:telescope} from $0$ to $k-1$, we have that for all $k\ge 1$, 
\begin{align}
\psi_{k}^*=\psi_{k}(x^{k})&\ge \psi_0(x^0) + \sum_{i=0}^{k-1}\,\alpha_i P(x^i)  -  \frac{2\diam_{\normt{\cdot}_*}(\calU)^2}{\mu_{\barcalS}}\sum_{i=0}^{k-1}\, \frac{\alpha_i^2}{{\beta_{i+1}}}. 
\end{align}
By the definitions of $\barx^k$ and $\tilx^k$ in~\eqref{eq:def_barx_tilx}, and the fact that $\beta_k = \sum_{i=0}^{k-1}\,\alpha_i$ (cf.~\ref{eq:step}),  we have 
\begin{align}
\textstyle \sum_{i=0}^{k-1}\,\alpha_i P(x^i)&\ge  \beta_k \max\{P(\barx^k),  P(\tilx^k)\}. 
\end{align}
This, combined with that $\psi_0\equiv 0$, yields 
\begin{equation}
\psi_{k}^*\ge \beta_k \max\{P(\barx^k),  P(\tilx^k)\}  -  \frac{2\diam_{\normt{\cdot}_*}(\calU)^2}{\mu_{\barcalS}}\sum_{i=0}^{k-1}\, \frac{\alpha_i^2}{{\beta_{i+1}}}. \label{eq:primal_ineq}
\end{equation}
 
Next, for $k\ge 0$, since $g^k \in\partial f(\rvA x^k)$, we have $\ipt{g^k}{\rvA x^k} = f(\rvA x^k) + f^*(g^k)$, and hence for $k\ge 1$,
\begin{align}
\psi_{k}^* = \textstyle\min_{x\in\bbX}\,\psi_k(x) &=\textstyle \min_{x\in\bbX}\, \textstyle\sum_{i=0}^{k-1} \alpha_i (\ipt{g^i}{\rvA x} - f^*(g^i)) + \beta_k h(x)\\
 & \textstyle=\min_{x\in\bbX}\, \beta_k (\ipt{\bars^k}{\rvA x}  +  h(x))  - \sum_{i=0}^{k-1} \alpha_i f^*(g^i) \\
 &= -\beta_k h^*({-\rvA^*\bars^k}) - \textstyle\sum_{i=0}^{k-1} \alpha_i f^*(g^i) \\
 &\le  -\beta_k (h^*({-\rvA^*\bars^k})    +  f^*(\bars^k)) \label{eq:conv_f*}\\
 & = -\beta_k D(\bars^k), \label{eq:dual_ineq} 
\end{align}
where in~\eqref{eq:conv_f*} we use the convexity of $f^*$. Combining~\eqref{eq:primal_ineq} and~\eqref{eq:dual_ineq}, we have
\begin{align}
\max\{P(\barx^k) + D(\bars^k),  P(\tilx^k)+D(\bars^k)\}\le \frac{2\diam_{\normt{\cdot}_*}(\calU)^2}{\mu_{\barcalS} \beta_k} \sum_{i=0}^{k-1}\, \frac{\alpha_i^2}{{\beta_{i+1}}}\le \frac{8\diam_{\normt{\cdot}_*}(\calU)^2}{\mu_{\barcalS}(k+1)}, \quad \forall\,k\ge 1,\nn
\end{align}
where in the last inequality we use~\ref{eq:step}. 
\end{proof}

\subsection{Some Remarks on Algorithm~\ref{algo:DA} and Theorem~\ref{thm:pdgap}}

Before concluding this section, let us make several remarks  regarding Algorithm~\ref{algo:DA} and Theorem~\ref{thm:pdgap}. 

\begin{itemize}[leftmargin = 0ex,label={},topsep=0pt,itemsep=5pt]
\item First, notice that 
the step-size sequences $\{\alpha_k\}_{k\ge 0}$ and $\{\beta_k\}_{k\ge 0}$ in~\ref{eq:step} do not depend on any problem parameters (such as $\diam_{\normt{\cdot}_*}(\calU)$ and $\mu_{\barcalS}$) that appear in the computational guarantees~\eqref{eq:comp_guarantee_DA}.

\item Second, note that $\diam_{\normt{\cdot}_*}(\calU)$ only depends on $\rvA$ and $f$ (or more precisely, $\dom f^*$), but not $h$.  In addition, since ${\max}_{y\in\calQ}\, \normt{y}\le  L$ (cf.~\cite[Corollary~13.3.3]{Rock_70}), where $L$ denotes the Lipschitz constant  of $f$,  $\diam_{\normt{\cdot}_*}(\calU)$ can indeed be bounded by $L$ as follows:
\begin{equation}
\diam_{\normt{\cdot}_*}(\calU)\le \normt{\rvA^*}\; {\max}_{y,y'\in\calQ}\, \normt{y-y'}_*\le 2\normt{\rvA}L,
\end{equation}
where $\normt{\rvA^*}$ denotes the operator norm of $\rvA^*$, and 
$$\normt{\rvA^*}= \textstyle\max_{\normt{y}=1, \normt{x}=1}\; \ipt{\rvA^*y}{x} = \max_{\normt{y}=1, \normt{x}=1}\; \ipt{\rvA x}{y} = \normt{\rvA}.$$ 
In some cases,  $\diam_{\normt{\cdot}_*}(\calU)$ can be significantly smaller than $2\normt{\rvA}L$, and thereby using $\diam_{\normt{\cdot}_*}(\calU)$ rather than  $2\normt{\rvA}L$ in~\eqref{eq:comp_guarantee_DA} provides a much tighter guarantee. 

\item Third, note that the constant $\mu_{\barcalS}$ appearing in~\eqref{eq:comp_guarantee_DA} depends on both $h$ and $f$. This is in contrast to the ``canonical'' case where $h$ is $\mu$-strongly convex (on its domain) --- in this case, $\mu_{\barcalS}=\mu$, 
which does not depend on $f$. 

\item Lastly, note that when $h$ is $\mu$-strongly convex, Grigas~\cite[Section~3.3.1]{Grigas_16} provides a computational guarantee of Algorithm~\ref{algo:DA} regarding the primal objective gap of~\ref{eq:P}. (The same is true for the computational guarantees of the mirror descent method for solving~\ref{eq:P}; cf.~\cite[Proposition~3.1]{Bach_15}.)  By replacing $\mu_{\barcalS}$ with $\mu$ in~\eqref{eq:comp_guarantee_DA}, 
we have indeed provided a computational guarantee regarding the primal-dual gap in this case, which is slightly stronger than the previous results. 
\end{itemize}

\section{Removing Assumptions on $f$ and $\rvA$, and a New DA-Type Method} \label{sec:new_DA}

Note that Assumption~\ref{assum:Q} involves both functions $h$ and $f$ and the linear operator $\rvA$, and it plays an important role in ensuring the well-definedness of the original DA method in Algorithm~\ref{algo:DA} (cf.\ Remark~\ref{rmk:well_defined}), as well as in analyzing the convergence rate of Algorithm~\ref{algo:DA}. 
In a sense, the success of Algorithm~\ref{algo:DA} with non-strongly-convex prox-function $h$ requires 
the ``good behavior'' of $f$ and $\rvA$, which can be somewhat restrictive. It is therefore natural to ask if we can completely remove any assumption on $f$ and $\rvA$. 
In this section, we shall demonstrate that this is indeed possible, provided that $\dom h^*$ is assumed to be open.
In fact,  one simple yet representative problem instance in this case is~\eqref{P-toy} with nonnegative data matrix $A$, as introduced in Section~\ref{sec:intro}. However,  
since 
Algorithm~\ref{algo:DA} may not be even well-defined in this case (cf.\ Remark~\ref{rmk:well_defined}), 
we need to develop new DA-type methods that work under the new assumptions. Indeed, based on the idea of ``dual monotonicity'', we shall propose a new DA-type method,  and show that it 
has an $O(1/k)$ convergence rate in terms of the primal-dual gap.  

Unlike the original DA method that was developed to solve the primal problem~\ref{eq:P}, the development of our new DA-type method will be primarily based on solving the dual problem~\ref{eq:D}. Before presenting our new algorithm, let us first  
observe that by~\cite[Corollary~31.2.1]{Rock_70}, 
strong duality holds between~\ref{eq:P} and~\ref{eq:D} (i.e.,  $P_* = -D_*$), and since $P_*>-\infty$, we know that $D_*<+\infty$ and hence~\ref{eq:D} is feasible, namely 
\begin{equation}
\dom D:= \{y\in\dom f^*: -\rvA^* y \in \dom h^*\}\ne \emptyset.  \label{eq:def_domD}
\end{equation}
In addition, since $\dom f^*$ is bounded (cf.\ Lemma~\ref{lem:compact}), $\dom D$ is bounded.

\subsection{Introduction to Algorithm~\ref{algo:MDA}} \label{sec:intro_algo_2}

Our new DA-type  method is shown in Algorithm~\ref{algo:MDA}. In fact, in terms of the 
structure and parameter choices, this new method is similar to the original one (i.e., Algorithm~\ref{algo:DA}), but the key difference is that we only generate the dual iterates $\{\bars^k\}_{k\ge 0}$ that keep or improve the dual objective value. Specifically, at each iteration $k\ge 0$, the iterate $\hats^k$ can be interpreted as the ``trial iterate'', and we only accept it if it results in a strict decrease of the dual objective value. 
For ease of reference, we shall call an iteration $k$ {\em ``active''} if $D(\hats^{k}) < D(\bars^k)$, and {\em ``idle''} otherwise. 
  Apart from this, another (somewhat subtle) difference between Algorithms~\ref{algo:DA} and~\ref{algo:MDA} lies the choice of initial dual iterate $\bars^0$. 
  Specifically, in Algorithm~\ref{algo:DA}, as 
defined in~\eqref{eq:def_bars}, $\bars^0 := g^{-1}  \in \partial f(\rvA x^{-1})$ for some $x^{-1}\in \bbX$, and under Assumption~\ref{assum:Q}, we know that $\bars^0\in \dom D$. 
However, when Assumption~\ref{assum:Q} fails to hold, such an initialization need not ensure that $\bars^0\in \dom D$, and therefore we need to specify a dually feasible initial iterate $\bars^0$ in Algorithm~\ref{algo:MDA}.

Let us make two remarks about Algorithm~\ref{algo:MDA}. First, note that Algorithm~\ref{algo:MDA} requires the zeroth-order oracle of the dual objective function $D$, which is not required in Algorithm~\ref{algo:DA}. Indeed, in most applications, the functions $f$ and $h$ have relatively simple forms, and therefore their Fenchel conjugates $f^*$ and $h^*$ can be easily found (and evaluated), which give rise to the zeroth-order oracle of $D$. That said, the well-definedness and analysis of Algorithm~\ref{algo:DA} require 
Assumption~\ref{assum:Q} to hold, which in turn requires the knowledge of 
$\dom f^*$ and $\dom h^*$. As mentioned in Remark~\ref{rmk:verify}, in some situations, finding $\dom f^*$ and $\dom h^*$ can be easier than finding 
$f^*$ and $h^*$ themselves. 
Second, in Algorithm~\ref{algo:MDA}, we only solve the sub-problem  in~\eqref{eq:subprob} and compute a subgradient of $f$ at an ``active'' iteration $k$. 
In contrast, we perform these two tasks at every iteration in Algorithm~\ref{algo:DA}. 

To ensure the well-definedness of Algorithm~\ref{algo:MDA} and analyze its convergence rate, we impose the following assumption on $h$.

\begin{assump}\label{assum:open_dom}
$\dom h^*$ is open. 
\end{assump}

\begin{remark}
Note that to verify Assumption~\ref{assum:open_dom}, we need not explicitly find $h^*$. In Section~\ref{sec:assump_dom_open}, we will provide some sufficient conditions on $h$ to ensure that Assumption~\ref{assum:open_dom} holds, based on the notion of {\em affine attainment}, along with illustrating examples.
\end{remark}

 Next, let us show that Algorithm~\ref{algo:MDA} is well-defined under Assumptions~\ref{assum:h} and~\ref{assum:open_dom}. To that end, given the initial dual iterate  $\bars^0\in\dom D$ in Algorithm~\ref{algo:MDA},  
define the sub-level set 
\begin{equation}
\calL:= \{y\in \bbY: D(y)\le D(\bars_0)\}\subseteq \dom D. \label{eq:def_L}
\end{equation}
Since $D$ is proper, closed and convex and has  bounded domain, we know that $\calL$ is nonempty, convex and compact. Based on $\calL$, we define 
\begin{equation}
\bar\calU:= -\rvA^*(\calL) \subseteq \dom h^*, 
\end{equation}
and we know that $\bar\calU$ is nonempty, convex and compact. Based on these definitions, let us show that  Algorithm~\ref{algo:MDA} is well-defined. 

\begin{algorithm}[t!]
\caption{Dual Averaging With Dual Monotonicity}\label{algo:MDA}
\begin{algorithmic}
\State {\bf Input}: 
$\bars^0\in\dom D$, step-size sequences $\{\alpha_k\}_{k\ge 0}$ and $\{\beta_k\}_{k\ge 0}$ chosen as in~\ref{eq:step} \phantom{r} \vspace{1ex}
\State {\bf Pre-start}: Compute 
 $x^0:=\argmin_{x\in\bbX}\; \ipt{\bars^0}{\rvA x} + h(x)$ and $g^{0}\in\partial f(\rvA x^{0})$ \vspace{1ex} 
\State {\bf At iteration $k\ge 0$}:
\begin{enumerate}[leftmargin = 6ex]
\item \label{item:def_hatsk} Compute $\hats^{k}:= (1-\tau_k)  \bars^k + \tau_k g^k$, where $\tau_k:= {\alpha_k }/\beta_{k+1}$
\item If $D(\hats^{k}) < D(\bars^k)$ then 
\begin{enumerate}[label = \roman*),leftmargin = 6ex]
\item $\bars^{k+1} = \hats^{k}$
\item $x^{k+1}:=  \argmin_{x\in \bbX}\; \ipt{\bars^{k+1}}{\rvA x}+ h(x)$ 
\item $g^{k+1} \in\partial f(\rvA x^{k+1})$
\end{enumerate}
Else   
\begin{enumerate}[label = \roman*),leftmargin = 6ex]
\item $\bars^{k+1} = \bars^{k}$
\item $x^{k+1}:=  x^k$ 
\item $g^{k+1} := g^k$
\end{enumerate}
\end{enumerate}
\end{algorithmic}
\end{algorithm}


\begin{lemma} \label{lem:well_posed_2}
Under Assumptions~\ref{assum:h} and~\ref{assum:open_dom}, define 
\begin{equation}
\calR:= \conv(\nabla h^*(\bar\calU)) . 
\end{equation}
In Algorithm~\ref{algo:MDA}, $\bars^k\in \calL$, $-\rvA^*\bars^{k}\in  \bar\calU$ and $x^{k}= \nabla h^*(-\rvA^*\bars^{k})\in\calR$ for $k\ge 0$. 
\end{lemma}

\begin{proof}
Note that in Algorithm~\ref{algo:MDA}, 
we always have 
\begin{equation}
x^{k}:=  {\argmin}_{x\in \bbX}\; \ipt{\bars^{k}}{\rvA x} + h(x), \qquad \forall\, k\ge 0.  \label{eq:def_x_bars2}
\end{equation}
Also, for all $k\ge 0$, we have 
\begin{equation}
D(\bars^{k+1}) = \min\{D(\bars^{k}), D(\hats^{k})\}\le D(\bars^k), \label{eq:D_monotone}
\end{equation}
and hence $\bars^k\in \calL$ for $k\ge 0$, implying that $-\rvA^*\bars^{k}\in  \bar\calU\subseteq \dom h^*$ for $k\ge 0$.
 Since $\dom h^*$ is open, we have $\dom h^* = \inter\dom h^*$, and  by Lemma~\ref{lem:h^*_diff}, we know that $h^*$ is differentiable at $-\rvA^*\bars^{k}$. 
 By~\eqref{eq:def_x_bars2}, 
we have $x^{k} = \nabla h^*(-\rvA^*\bars^{k})$ and since $-\rvA^*\bars^{k}\in  \bar\calU$, we have  $x^{k} \in\calR$.
\end{proof}


\subsection{Convergence Rate Analysis of Algorithm~\ref{algo:MDA}}

The lemma below establishes the smoothness property of $h^*$ on any nonempty convex compact set inside $\dom h^*$, which is crucial in our analysis of Algorithm~\ref{algo:MDA}. 

\begin{lemma} \label{lem:h^*}
Under Assumption~\ref{assum:h}, for any nonempty, convex and compact set $\calW\subseteq \inter\dom h^*$, define 
\begin{equation}
\calS:= \conv(\nabla h^*(\calW)) .  \label{eq:def_S2}
\end{equation}
Then $\calS$ is nonempty, convex and compact, and $\calS \subseteq \dom h$. In addition, the function $h^*$ is $\mu_{\calS}^{-1}$-smooth on $\calW$, namely
\begin{equation}
\normt{\nabla h^*(u_1) - \nabla h^*(u_2)}\le \mu_{\calS}^{-1} \normt{u_1 - u_2}_*, \quad\forall\, u_1,u_2\in \calW, \label{eq:grad_ineq0}
\end{equation}
 where $\mu_\calS>0$ is defined in~\eqref{eq:def_muS}.  
In particular, 
 if Assumption~\ref{assum:open_dom} also holds, then the above holds for any $\calW\subseteq \dom h^*$. 
\end{lemma}

\begin{proof}

Since $\calW\subseteq\inter\dom h^*$, using the same argument as in the proof of Lemma~\ref{lem:compact}, we know that $\calS \subseteq \dom h$ is nonempty, convex and compact. Now, take any $u_1,u_2\in\calW$. For $i =1,2$, since $u_i\in \inter\dom h^*$, we know that i) $h^*$ is differentiable at $u_i$, and ii) $g_i := h - \ipt{u_i}{\cdot}$ is coercive (cf.~\cite[Fact 2.11]{Bauschke_97}). This, together with  Lemma~\ref{lem:h^*_diff}, shows that $g_i$ has a unique minimizer on $\bbX$, which allows us  to define 
\begin{equation}
x_i^* := {\argmin}_{x\in\bbX}\; h(x) - \ipt{x}{u_i}\in \dom h. \label{eq:x^*_i} 
\end{equation}
By the optimality condition of~\eqref{eq:x^*_i}, we know that $u_i\in \partial h(x_i^*)$  and hence $x^*_i:=\nabla h^*(u_i) \in\calS$, for $i =1,2$. By the $\mu_\calS$-strong convexity of  $h$ on $\calS$, we have 
\begin{equation}
 \normt{x_1^*-x_2^*} \normt{u_1-u_2}_*\ge \ipt{u_1-u_2}{x_1^*-x_2^*}\ge \mu_\calS\normt{x_1^*-x_2^*}^2.
\end{equation}
If $x_1^*\ne x_2^*$, we then have 
\begin{equation}
\normt{\nabla h^*(u_1)-\nabla h^*(u_2)}=\normt{x_1^*-x_2^*}\le \mu_\calS^{-1} \normt{u_1-u_2}_*. \label{eq:grad_ineq}
\end{equation}
If $x_1^*= x_2^*$, then~\eqref{eq:grad_ineq}  trivially holds. 
Finally, under Assumption~\ref{assum:open_dom}, we know that $\dom h^*$ is open, and hence $ \dom h^* = \inter\dom h^*$. Therefore, the above holds for any $\calW\subseteq \dom h^*$. 
\end{proof}

\begin{remark}[Replacing Assumption~\ref{assum:h} with other assumptions] 
As we shall see later, the role that Assumption~\ref{assum:h} plays in analyzing Algorithm~\ref{algo:MDA} is  ensuring that 
$h^*$ is smooth on any nonempty, convex and compact set $\calW\subseteq \inter\dom h^*$. As such, Assumption~\ref{assum:h} can be replaced with any other assumption that ensures the same  condition. 
 For example, we can replace Assumption~\ref{assum:h} with $h$ being very strictly convex and Legendre. In fact, by the same reasoning as in the proof of Lemma~\ref{lem:vsc_suff}, we know that $\calS$ in~\eqref{eq:def_S2} is  nonempty, convex and compact, $\calS\subseteq  \inter\dom h$ and $h$ is $\mu_{\calS}$-strongly-convex on $\calS$. 
 By the same reasoning  as in the proof of Lemma~\ref{lem:h^*}, 
 we can still show that $h^*$ is $\mu_{\calS}^{-1}$-smooth on $\calW$, where  $\mu_\calS>0$ is defined in~\eqref{eq:def_muS}.  
\end{remark}

In the analysis of Algorithm~\ref{algo:MDA}, we also need the following technical lemma. 

\begin{lemma}\label{lem:seq}  
Given $A\ge 0$ and $k_0\ge 0$, suppose that $\{a_k\}_{k\ge 0}$ and $\{b_k\}_{k\ge 0}$ are two nonnegative sequences  satisfying that 
\begin{equation}
a_k \le b_k\quad \andd\quad a_{k+1}\le a_k - \tau_k b_k + (A/2) \tau_k^2, \quad \forall\, k\ge k_0, \label{eq:recur}
\end{equation}
where $\tau_k :=\alpha_k/\beta_{k+1}$ for $k\ge k_0$, and $\{\alpha_k\}_{k\ge 0}$ and $\{\beta_k\}_{k\ge 0}$ are chosen as in~\ref{eq:step}.   
Then we have 
\begin{align}
a_k&\le \frac{{k_0}(k_0+1)a_{k_0} + 2A(k-k_0)}{k(k+1)}, \quad \forall\,k\ge k_0+1, \label{eq:a_k} \quad\andd \\
{\min}_{i=\lfloor (k+k_0)/2 \rfloor}^{k-1}\; b_i &\le
 \frac{12(k_0+1)^2 }{(k - k_0)(k+k_0)} a_{k_0}+ 
\frac{26A}{k+k_0}, \quad \forall\, k\ge k_0 + 1. \label{eq:b_k} 
\end{align}
\end{lemma}

\begin{proof}
See Appendix~\ref{app:proof_seq}. 
\end{proof}


Our convergence rate analysis of Algorithm~\ref{algo:MDA} is geometric in nature. Due to this, let us define a few geometric quantities. 
First, let us define 
\begin{equation}
\Delta := \dist_{\normt{\cdot}_*}(\bar\calU,\, \bdry \dom h^*), 
\label{eq:def_Delta} 
\end{equation}
if $\dom h^*\subsetneqq \bbX^*$ (i.e., $\bdry \dom h^*\ne \emptyset$), 
and $\Delta := +\infty$ if $\dom h^*= \bbX^*$. Note that under Assumption~\ref{assum:open_dom}, we always have $\Delta>0$,  as shown in the lemma below. 

\begin{lemma}\label{lem:Delta}
Under Assumption~\ref{assum:open_dom}, if $\dom h^*\subsetneqq \bbX^*$, then there exist $u\in\bar\calU$ and $u'\in \bdry \dom h^* $ such that $\Delta = \normt{u-u'}_*>0$. 
\end{lemma}

\begin{proof}
See Appendix~\ref{app:proof_Delta}. 
\end{proof}

Next, for any $r\ge 0$, let us define the $r$-enlargement of the set $\bar\calU$ as 
\begin{equation}
\bar\calU(r) := \{u\in\bbX^*: \dist_{\normt{\cdot}_*}(u,\bar\calU)\le r\}.
\end{equation}
Indeed, $\bar\calU(r)$ has some nice geometric properties, which are stated in the lemma below. 

\begin{lemma} \label{lem:Ur}
For any $r\ge 0$, $\bar\calU(r)$ is nonempty, convex and compact. Additionally, under Assumption~\ref{assum:open_dom}, we have $\bar\calU(r)\subseteq\dom h^*$ for any $0\le r < \Delta$. 
\end{lemma}

\begin{proof}
See Appendix~\ref{app:proof_Ur}. 
\end{proof}

Now, let us fix any $0\le r < \Delta$, and let $\calV(r)$ be a convex and compact set such that 
\begin{equation}
\bar\calU(r) \subseteq \calV(r)\subseteq \dom h^*.  \label{eq:V_r}
\end{equation}
By Lemma~\ref{lem:Ur}, 
one obvious choice of $\calV(r)$ is $\bar\calU(r)$, but other choices of $\calV(r)$ may exist as well. 
By Lemma~\ref{lem:h^*}, we know that 
 $h^*$ is $\mu_{\calS(r)}^{-1}$-smooth on $\calV(r)$, where 
\begin{equation}
 \calS(r) := \conv(\nabla h^*(\calV(r)))\subseteq \dom h. \label{eq:S_r}
 \end{equation}
In addition, recall 
that $\calU:=-\rvA^*(\calQ)$ for $\calQ:=\cl \dom f^*$ in~\eqref{eq:def_QU}. Based on $\calU$, $\bar\calU$ and $\calV(r)$, 
define 
\begin{equation}
K_{\calV(r)}:= \min\{k\ge 0: \; (1-\tau_k) u + \tau_k u' \in \calV(r), \quad \forall\,u\in \bar\calU, \; \forall\, u'\in \calU \} . \label{eq:def_Kr}
\end{equation}
It turns out that $K_{\calV(r)}$ is well-defined for any $0< r < \Delta$, and as shown in the lemma below, it admits a simple upper bound in terms of  $r$ and the ``furthest distance'' between $\bar\calU$ and $\calU$, namely
\begin{equation}
\ell_{\normt{\cdot}_*}(\bar\calU,\calU):= \max\{\normt{u-u'}_* :u\in\bar\calU, \, u'\in\calU \} .  \label{eq:def_ell}
\end{equation}

\begin{lemma} \label{lem:barU_U}
We have $\bar\calU\subseteq\calU$ and hence 
\begin{equation}
\diam_{\normt{\cdot}_*}(\calU)/2 \le \ell_{\normt{\cdot}_*}(\bar\calU,\calU)\le \diam_{\normt{\cdot}_*}(\calU). \label{eq:diam_ell} 
\end{equation} 
Under  Assumption~\ref{assum:open_dom}, for any $0< r < \Delta$, we have 
\begin{equation}
K_{\calV(r)}\le K_{\bar\calU(r)}\le  2\left\lceil \big(\ell_{\normt{\cdot}_*}(\bar\calU,\calU)/r - 1\big)_+\right\rceil,  \label{eq:K_r_ub}
\end{equation}
where $a_+:= \max\{a,0\}$. 

\end{lemma}

\begin{proof}
For notational brevity, let $\ell:= \ell_{\normt{\cdot}_*}(\bar\calU,\calU)$. Since $\bar\calU(r)\subseteq \calV(r)$, we clearly have $K_{\calV(r)}\le K_{\bar\calU(r)}$. 
By~\eqref{eq:def_L},~\eqref{eq:def_domD} and
~\eqref{eq:def_QU}, we have $\calL\subseteq \dom D\subseteq \dom f^*\subseteq \calQ$, and hence $\bar\calU\subseteq\calU$. As a result, $\ell\le \diam_{\normt{\cdot}_*}(\calU)$. In addition, for any $\baru\in\bar\calU$, 
\begin{align}
\textstyle\diam_{\normt{\cdot}_*}(\calU)  = {\max}_{u,u'\in \calU}\; \normt{u-u'}_* &\le {\max}_{u,u'\in \calU}\; \normt{u-\baru}_* + \normt{u'-\baru}_* \le 2\ell. 
\end{align} 
This proves~\eqref{eq:diam_ell}. Next, we prove~\eqref{eq:K_r_ub} by considering two cases. 
If $r\ge \ell$, then for any $u\in\bar\calU$ and $u'\in\calU $, we have 
\begin{equation}
\dist_{\normt{\cdot}_*}(u',\bar\calU)\le \normt{u-u'}_* \le \ell\le r,
\end{equation}
and hence $u'\in \bar\calU(r) \subseteq \calV(r)$. Since $\tau_0 = 1$, for any $u\in\bar\calU$ and $u'\in\calU $,  we have  $(1-\tau_0) u + \tau_0 u' = u' \in \calV(r),$ and hence $K_{\calV(r)} = 0$. If $r< \ell$,
then let $k := 2\lceil \ell/r - 1\rceil$, and hence $\tau_k = 2/(k+2)\le r/\ell$. As a result, for any $u\in\bar\calU$ and $u'\in\calU $, we have 
\begin{equation}
\dist_{\normt{\cdot}_*}\big((1-\tau_k) u + \tau_k u',\bar\calU\big) \le \normt{(1-\tau_k) u + \tau_k u' - u}_* = \tau_k \normt{u' - u}_*\le (r/\ell) \ell  = r, 
\end{equation} 
and hence $(1-\tau_k) u + \tau_k u'\in \bar\calU(r) \subseteq \calV(r)$. Therefore, $K_{\calV(r)}\le k$ and we complete the proof. 
\end{proof}

\begin{remark}
Note that depending on the geometry of $\bar\calU$ and $\dom h^*$, for any  $0\le r< \Delta$, the set $\calV(r)$ can be chosen to be much larger than $\bar\calU(r)$, and hence $K_{\calV(r)}$ can be potentially much smaller than $K_{\bar\calU(r)}$.
As a simple example, let $\bbX^* := (\bbR^n,\normt{\cdot}_{\infty})$, $\dom h^* = (0,1)^n$ and  $\bar\calU = [\epsilon,\,2\epsilon]^n $, where $\normt{\cdot}_{\infty}$ denotes the $\ell_\infty$-norm and $\epsilon>0$ is small. For any $0\le r<\epsilon$, by definition, we have $\bar\calU(r) = [\epsilon-r,\,2\epsilon+r]^n $. In this case, we can choose   $\calV(r) = [a,\,b]^n $ for any $0<a \le  \epsilon-r$ and $2\epsilon +r\le b<1$, which can be much larger than $\bar\calU(r)$. 
\end{remark}

Equipped with all the preparatory results above, we are now ready to state the convergence rate of Algorithm~\ref{algo:MDA} under  Assumptions~\ref{assum:h} and~\ref{assum:open_dom}. 

\begin{theorem} \label{thm:analysis_algo2}
Fix any $0< r < \Delta$. Let $\calV(r)$ be a convex and compact set that satisfies~\eqref{eq:V_r}, and $K_{\calV(r)}$ be defined in~\eqref{eq:def_Kr}. 
Under  Assumptions~\ref{assum:h} and~\ref{assum:open_dom}, in Algorithm~\ref{algo:MDA}, we have 
\begin{equation}
D(\bars^k)\le D(\bars^0), \quad \forall\, 1\le k \le K_{\calV(r)}, 
\end{equation}
and $K_{\calV(r)}$ is upper bounded  in~\eqref{eq:K_r_ub}. In addition, for all $k\ge K_{\calV(r)}+1$,  we have 
\begin{align}
{\min}_{i=0}^{k-1}\; P(x^i) + D(\bar{s}^i) &\le  \frac{12(K_{\calV(r)}+1)^2 }{(k - K_{\calV(r)})(k+K_{\calV(r)})} \big(D(\bars^{K_{\calV(r)}}) - D_*\big) + \frac{26 \ell_{\normt{\cdot}_*}(\bar\calU,\calU)^2}{\mu_{\calS(r)}(k+K_{\calV(r)})}, \label{eq:algo2_rate}
\end{align}
where $\ell_{\normt{\cdot}_*}(\bar\calU,\calU)$ and ${\calS(r)}$ are defined in~\eqref{eq:def_ell} and~\eqref{eq:S_r}, respectively. In addition, 
let 
\begin{equation}
\hatx^k\in {\argmin}_{x\in\{x_0, \ldots\,x_{k}\}}\; P(x),\quad \forall\,  k\ge 0,  \label{eq:def_hatx}
\end{equation}
  then we have that for all $k\ge K_{\calV(r)}+1$,  
\begin{equation}
 P(\hatx^{k}) +  D(\bar{s}^k) \le \frac{{K_{\calV(r)}}(K_{\calV(r)}+1)}{k(k+1)} \left(P(\hatx^{K_{\calV(r)}}) +  D(\bar{s}^{K_{\calV(r)}})\right) + \frac{2\ell_{\normt{\cdot}_*}(\bar\calU,\calU)^2\big(k-K_{\calV(r)}\big)}{\mu_{\calS(r)} \, k(k+1)}. \label{eq:algo2_rate1.5} 
\end{equation}
\end{theorem}

\begin{proof}
By Lemma~\ref{lem:well_posed_2},  we know that for all $k\ge 0$, $\bars^k\in\calL$ and 
 $-\rvA^* \bars^k \in \barcalU\subseteq \calV(r)$. 
For $k\ge 0$, since $g^k\in\calQ$, we have  $-\rvA^* g^k\in\calU$. Thus   by the definition of $K_{\calV(r)}$ in~\eqref{eq:def_Kr}, 
 we have 
\begin{equation}
-\rvA^*(\bars^k +    \tau_{K_{\calV(r)}} (g^k -\bars^k))\in \calV(r), \quad \forall\, k\ge 0. 
\end{equation} 
Since $\{\tau_k\}_{k\ge 0}$ is monotonically decreasing, for all $k\ge K_{\calV(r)}$, we have ${\tau_k}/{\tau_{K_{\calV(r)}}} \in (0, 1)$. Since 
\begin{equation}
\hats^k = \bars^k + \tau_k (g^k -\bars^k) = (1- {\tau_k}/{\tau_{K_{\calV(r)}}} )  \bars^k + ({\tau_k}/{\tau_{K_{\calV(r)}}}) (\bars^k +    \tau_{K_{\calV(r)}} (g^k -\bars^k)),
\end{equation}
and $\calV(r)$ is convex, we have $-\rvA^*\hat{s}^k\in\calV(r)$. From Lemma~\ref{lem:h^*}, we know that $h^*$ is $\mu_{\calS(r)}^{-1}$-smooth on $\calV(r)$, and hence for all $k\ge K_{\calV(r)}$,
\begin{align}
\hspace{-.8em}  h^*(-\rvA^*\hat{s}^k) &\le h^*(-\rvA^*\bar{s}^k) -\ipt{\nabla h^*(-\rvA^*\bars^{k})}{\rvA^*(\hat{s}^k -\bars^k)} +  \frac{\normt{\rvA^*(\hat{s}^k -\bars^k)}_*^2}{2\mu_{\calS(r)}}\\
&\le h^*(-\rvA^*\bar{s}^k) - \tau_k\ipt{\rvA x^k}{g^k -\bars^k} +  \tau_k^2\frac{\normt{\rvA^*(g^k -\bars^k)}_*^2}{2\mu_{\calS(r)}}\\
&\le h^*(-\rvA^*\bar{s}^k) - \tau_k\big(h^*(-\rvA^*\bars^k) + h(x^k) + f^*(g^k) + f(\rvA x^k)\big) + \tau_k^2\frac{\ell_{\normt{\cdot}_*}(\bar\calU,\calU)^2}{2\mu_{\calS(r)}} \label{eq:algo2_h*}
\end{align}
where we use $x^{k}= \nabla h^*(-\rvA^*\bars^{k})$ (cf.\ Lemma~\ref{lem:well_posed_2}), $g^k\in \partial f(\rvA x^k)$ 
 and the definition of $\ell_{\normt{\cdot}_*}(\bar\calU,\calU)$ in~\eqref{eq:def_ell}. Also, by the convexity of $f^*$, we have 
\begin{align}
f^*(\hat{s}^k) \le (1-\tau_k) f^*(\bar{s}^k) + \tau_k f^*(g^k) = f^*(\bar{s}^k) - \tau_k(f^*(\bar{s}^k) - f^*(g^k)). \label{eq:algo2_f*}
\end{align}
Combining~\eqref{eq:algo2_h*} and~\eqref{eq:algo2_f*}, and use~\eqref{eq:D_monotone},   we have that for all $k\ge K_{\calV(r)}$, 
\begin{align}
D(\bars^{k+1}) - D_* \le D(\hat{s}^k) - D_* \le (D(\bar{s}^k)  - D_*) - \tau_k (P(x^k) + D(\bar{s}^k) ) +  \tau_k^2\frac{\ell_{\normt{\cdot}_*}(\bar\calU,\calU)^2}{2\mu_{\calS(r)}}. \label{eq:D_bars}
\end{align}
Since $P(x^k)\ge P_* = -D_*$, we have $P(x^k) + D(\bar{s}^k)\ge D(\bar{s}^k)  - D_*$, and hence we can invoke~\eqref{eq:b_k} in Lemma~\ref{lem:seq} to obtain~\eqref{eq:algo2_rate}. In addition, by the definition of $\hatx^{k}$, we have
\begin{equation}
P(\hatx^{k+1}) - P_*\le P(\hatx^{k}) - P_*, \quad \forall\, k\ge 0.  \label{eq:P_hatx}
\end{equation}
 Since $P_* = -D_*$, combining~\eqref{eq:D_bars} and~\eqref{eq:P_hatx},  we have 
\begin{align}
P(\hatx^{k+1}) + D(\bars^{k+1}) &\le (P(\hatx^{k}) +  D(\bar{s}^k)) - \tau_k (P(x^k) + D(\bar{s}^k) ) +  \tau_k^2\frac{\ell_{\normt{\cdot}_*}(\bar\calU,\calU)^2}{2\mu_{\calS(r)}},\quad \forall\, k\ge K_{\calV(r)}.\nn 
\end{align}
Since $P(x^k) + D(\bar{s}^k)\ge P(\hatx^{k}) +  D(\bar{s}^k)$ for $k\ge 0$, we can invoke~\eqref{eq:a_k} in Lemma~\ref{lem:seq} and arrive at~\eqref{eq:algo2_rate1.5}. 
\end{proof}

\begin{remark}[Interpreting Theorem~\ref{thm:analysis_algo2}] \label{rmk:interpret_algo2}
From Theorem~\ref{thm:analysis_algo2}, we see that the analysis of the convergence rate of Algorithm~\ref{algo:MDA} is divided into two phases. In the first phase (i.e., $1\le k \le K_{\calV(r)}$), 
we are not able to  provide convergence rate guarantees on the dual objective gap or the primal-dual gap, except that the the dual objective gap does not increase. This is because 
it could be the case that the ``trial iterate'' $\hats^k\not\in \dom D$ for any $0\le k < K_{\calV(r)}$ (and hence $\bars^k = \bars^0$ for $1\le k \le K_{\calV(r)}$), and in this case, we have no information on the current dual objective value $D(\bars^0)$. 
However, in the second phase (i.e., $k\ge  K_{\calV(r)}$), 
we know that $-\rvA^*\hats^k\in \calV(r)$, and by the $\mu_{\calS(r)}^{-1}$-smoothness of $h^*$ on $\calV(r)$, we can upper bound $D(\hats^k)$   in a concrete way (cf.~\eqref{eq:D_bars}). 
Since $D(\bars^{k+1}) = \min\{D(\bars^{k}), D(\hats^{k})\}\le D(\hats^{k})$, this provides an upper bound on $D(\bars^{k+1})$ as well, which in turn allows us to derive the convergence rate of (various forms of) the primal-dual gap 
for $k\ge  K_{\calV(r)}+1$.

\begin{remark}[Iteration and Oracle Complexities of Algorithm~\ref{algo:MDA}] \label{rmk:complexity}
From~\eqref{eq:algo2_rate} and Lemma~\ref{lem:barU_U},  some simple algebra reveal that to achieve an $\varepsilon$-primal-dual gap, the number of iterations needed by Algorithm~\ref{algo:MDA} is of order 
\begin{equation}
O\left(\max\left\{\frac{\ell_{\normt{\cdot}_*}(\bar\calU,\calU)}{\Delta}\sqrt{\frac{D(\bars^0) - D_*}{\varepsilon}}, \; \frac{\ell_{\normt{\cdot}_*}(\bar\calU,\calU)^2}{\mu_{\calS(r)}\,\varepsilon}\right\}\right), \label{eq:iter_comp_algo2}
\end{equation}
where $\Delta$ is defined in~\eqref{eq:def_Delta}. As mentioned in Section~\ref{sec:intro_algo_2}, Algorithm~\ref{algo:MDA} uses three types of oracles, namely ($\calO_1$) the zeroth-order oracle of the dual objective function $D$, ($\calO_2$) the sub-problem minimization oracle associated with $h$ (cf.~\eqref{eq:subprob}) and ($\calO_3$) the first-order oracle of $f$. 
Note that in the first $K$ iterations of  Algorithm~\ref{algo:MDA}, the number of oracle calls of $\calO_1$
is clearly $K$. 
In contrast, the number of  oracle calls of $\calO_2$ and $\calO_3$ is equal to the number of ``active'' iterations within the first $K$ iterations, which we denote by $K_{\rm act}$. For some problem instances, $K_{\rm act}$ may be much lower than  $K$, however, this may not be the case in general. 
\end{remark}

\end{remark}

\begin{remark}[Different forms of the primal-dual gap] \label{rmk:two_duality_gaps}
Note that for $k\ge K_{\calV(r)}+1$, Theorem~\ref{thm:analysis_algo2} provides the convergence rates of two forms of the primal-dual gap,  namely ${\min}_{i=0}^{k-1}\; P(x^i) + D(\bar{s}^i)$ and ${\min}_{i=0}^{k}\; P(x^i) + D(\bar{s}^k)$. Since $\{D(\bar{s}^k)\}_{k\ge 0}$ is monotone, we have 
\begin{equation}
{\min}_{i=0}^{k}\; P(x^i) + D(\bar{s}^i)\ge {\min}_{i=0}^{k}\; P(x^i) + {\min}_{i=0}^{k}\;D(\bar{s}^i) = {\min}_{i=0}^{k}\; P(x^i) + D(\bar{s}^k),
\end{equation}
and therefore, the convergence rate in~\eqref{eq:algo2_rate} (with $k$ replaced by $k+1$) is also valid for ${\min}_{i=0}^{k}\; P(x^i) + D(\bar{s}^k)$. As a result, we can take the  convergence rate of ${\min}_{i=0}^{k}\; P(x^i) + D(\bar{s}^k)$ 
to be the minimum of the rates in~\eqref{eq:algo2_rate1.5} and~\eqref{eq:algo2_rate}  (with $k$ replaced by $k+1$). 
\end{remark}

Next, a natural question one may have is whether Algorithm~\ref{algo:MDA} also works in the setting of Section~\ref{sec:DA_analysis}, i.e., under Assumptions~\ref{assum:h} and~\ref{assum:Q}. The theorem below provides an affirmative answer. In fact, 
Algorithm~\ref{algo:MDA} shares similar computational guarantees to Algorithm~\ref{algo:DA} in this setting. 

\begin{theorem} \label{thm:conv_rate_algo2}
Under  Assumptions~\ref{assum:h} and~\ref{assum:Q}, in Algorithm~\ref{algo:MDA}, we have that 
for all $k\ge 1$,  
\begin{align}
{\min}_{i=0}^{k-1}\; P(x^i) + D(\bar{s}^i) &\le  \frac{12 }{k^2 } \big(D(\bars^0) - D_*\big) + \frac{26 \ell_{\normt{\cdot}_*}(\bar\calU,\calU)^2}{\mu_{\bar\calS}\,k}, \quad \andd\label{eq:algo2_rate2}\\
P(\hatx^{k}) +  D(\bar{s}^k) &\le  \frac{2\ell_{\normt{\cdot}_*}(\bar\calU,\calU)^2}{\mu_{\bar\calS} \, (k+1)}, \label{eq:algo2_rate2.5} 
\end{align}
where $\ell_{\normt{\cdot}_*}(\bar\calU,\calU)$, $\bar\calS$ and $\hatx^{k}$ are defined in~\eqref{eq:def_ell},~\eqref{eq:def_barS} and~\eqref{eq:def_hatx}, respectively. 
\end{theorem}

\begin{proof}
The proof follows the same line of reasoning as that of Theorem~\ref{thm:analysis_algo2}. First, note that under Assumption~\ref{assum:Q}, we have $\bar\calU\subseteq \calU\subseteq \inter \dom h^*$. As a result, we have $-\rvA^*\hat{s}^k\in\calU$ for all $k\ge 0$. From Lemma~\ref{lem:h^*}, we know that $h^*$ is $\mu_{\bar\calS}^{-1}$-smooth on $\calU$, 
and we deduce that  
\begin{align}
D(\bars^{k+1}) - D_*\le  (D(\bar{s}^k)  - D_*) - \tau_k (P(x^k) + D(\bar{s}^k) ) +  \tau_k^2\frac{\ell_{\normt{\cdot}_*}(\bar\calU,\calU)^2}{2\mu_{\bar\calS}}, \quad \forall\,k\ge 0.  
\end{align}
Invoking~\eqref{eq:b_k} in Lemma~\ref{lem:seq} with $k_0=0$, we arrive at~\eqref{eq:algo2_rate2}. 
Also, by~\eqref{eq:P_hatx} and $P_* = -D_*$, we  have
\begin{align}
P(\hatx^{k+1}) + D(\bars^{k+1}) \le  (P(\hatx^{k}) + D(\bar{s}^k))   - \tau_k (P(x^k) + D(\bar{s}^k) ) +  \tau_k^2\frac{\ell_{\normt{\cdot}_*}(\bar\calU,\calU)^2}{2\mu_{\bar\calS}}, \quad \forall\,k\ge 0.  
\end{align}
Invoking~\eqref{eq:a_k} in Lemma~\ref{lem:seq} with $k_0=0$, we arrive at~\eqref{eq:algo2_rate2.5}.
\end{proof}

\begin{remark}
Similar to Remark~\ref{rmk:two_duality_gaps}, the convergence rate in~\eqref{eq:algo2_rate2} (with $k$ replaced by $k+1$) also applies to ${\min}_{i=0}^{k}\; P(x^i) + D(\bar{s}^k)$. However, note that this rate is strictly inferior to the one in~\eqref{eq:algo2_rate2.5}. 
\end{remark}

\begin{remark}[Comparison between 
 Theorems~\ref{thm:pdgap} and~\ref{thm:conv_rate_algo2}]
Under Assumptions~\ref{assum:h} and~\ref{assum:Q}, Theorems~\ref{thm:pdgap} and~\ref{thm:conv_rate_algo2} provide the convergence rates of Algorithms~\ref{algo:DA} and~\ref{algo:MDA},  respectively. 
At a high level, Theorems~\ref{thm:pdgap} and~\ref{thm:conv_rate_algo2} indicate that both Algorithms~\ref{algo:DA} and~\ref{algo:MDA} converge at rate $O(1/k)$ in terms of 
the primal-dual gap. However, note that the convergence rates in these two theorems actually concern different forms of the primal-dual gaps, and also depend on different quantities. The differences arise from the different structures of Algorithms~\ref{algo:DA} and~\ref{algo:MDA}, as well as the different analytic approaches. Specifically, the analysis of Algorithm~\ref{algo:DA} mainly proceeds on the primal side, and is based on the sequence of auxiliary functions $\{\psi_k\}_{k\ge 0}$; in contrast, the analysis of Algorithm~\ref{algo:MDA} mainly proceeds on the dual side, and is based on the $\mu_{\bar\calS}^{-1}$-smoothness of $h^*$.  
\end{remark}

\begin{remark}[Align 
Theorem~\ref{thm:conv_rate_algo2} with 
Theorem~\ref{thm:pdgap}]
Note that the convergence rate of Algorithm~\ref{algo:DA} in Theorem~\ref{thm:pdgap} depends on two quantities, namely $\diam_{\normt{\cdot}_*}(\calU)$ and $\mu_{\barcalS}$. In contrast, the convergence rate~\eqref{eq:algo2_rate2} 
in Theorem~\ref{thm:conv_rate_algo2} involves three quantities, namely %
$D(\bars^0) - D_*$ (i.e., the initial dual objective gap),  $\ell_{\normt{\cdot}_*}(\bar\calU,\calU)$ 
and $\mu_{\bar\calS}$. To align the convergence rate result in Theorem~\ref{thm:conv_rate_algo2} with that in Theorem~\ref{thm:pdgap}, first note that $\ell_{\normt{\cdot}_*}(\bar\calU,\calU)\le \diam_{\normt{\cdot}_*}(\calU)$ by Lemma~\ref{lem:barU_U}. Next, if $\bars^0$ is chosen via a ``pre-start'' procedure, then $D(\bars^0) - D_*$ can be upper bounded by some quantity that depends on $\diam_{\normt{\cdot}_*}(\calU)$ and $\mu_{\bar\calS}$. 
Specifically, let $\bars^{-1}$ be any point in $\dom D$, $x^{-1}:= \argmin_{x\in\bbX}\; \ipt{\bars^{-1}}{\rvA x} + h(x)$ 
and $\bars^0\in \partial f(\rvA x^{-1})$. Since both $\bars^{-1}, \bars^0\in \dom f^*,$ we have both $-\rvA^*\bars^{-1}, -\rvA^*\bars^0\in \calU.$ Using the same proof of Lemma~\ref{lem:h^*}, we can show that under Assumptions~\ref{assum:h} and~\ref{assum:Q}, $h^*$ is $\mu_{\bar\calS}^{-1}$-smooth on $\calU$, and so we have 
\begin{align*}
h^*(-\rvA^*\bars^0)&\le h^*(-\rvA^*\bars^{-1}) - \ipt{\rvA x^{-1}}{\bars^0 - \bars^{-1}} + {\normt{\rvA^*(\bars^0- \bars^{-1})}^2_*}/({2\mu_{\bar\calS}})\\
&\le h^*(-\rvA^*\bars^{-1}) - \big(h^*(-\rvA^*s^{-1}) + h(x^{-1}) + f(\rvA x^{-1}) + f^*(\bars^0) \big) + {\diam_{\normt{\cdot}_*}(\calU)^2}/({2\mu_{\bar\calS}})\\
&\le -P(x^{-1}) - f^*(\bars^0)  + {\diam_{\normt{\cdot}_*}(\calU)^2}/({2\mu_{\bar\calS}}),
\end{align*}
where we use $x^{-1} = \nabla h^*(-\rvA^*\bars^{-1}) $ and $\bars^0\in \partial f(\rvA x^{-1})$. As a result, we have  
\begin{align*}
D(s^0) - D_*\le P(x^{-1})+D(s^0)\le {\diam_{\normt{\cdot}_*}(\calU)^2}/({2\mu_{\bar\calS}}). 
\end{align*}
As a result,~\eqref{eq:algo2_rate2} now becomes
\begin{align}
{\min}_{i=0}^{k-1}\; P(x^i) + D(\bar{s}^i) &\le  \frac{ \diam_{\normt{\cdot}_*}(\calU)^2}{\mu_{\bar\calS}}\left(\frac{6}{k^2} + \frac{13}{k} \right), \quad \forall\, k\ge 1.  
\end{align}
\end{remark}


\subsection{Certificates for Assumption~\ref{assum:open_dom} } \label{sec:assump_dom_open}

In this section we provide two conditions on $h$ that ensure $\dom h^*$ to be open, along with illustrating examples.  Let us start with two definitions.

\begin{definition}[{Recession Function;~\cite[Theorem~8.5]{Rock_70}}]
Given a proper, closed and convex function $h:\bbX\to\barbbR$, define its recession function $r_h:\bbX\to\barbbR$ as 
\begin{equation}
r_h(v) = {\sup}_{x\in\dom h}\;  h(x+v) - h(x), \quad \forall\, v\in \bbX.  \label{eq:def_r_h}
\end{equation}
In addition, $r_h$ is proper, closed, convex and positively homogeneous. 
\end{definition}

\begin{definition}[Affine Attainment] \label{def:aff_att}
A function  
$h:\bbX\to\barbbR$ is called {\em affine attaining} if for any $u\in\bbX^*$, if $g_u:= h - \ipt{u}{\cdot}$ is lower bounded, then $g_u$ has a minimizer on $\bbX$. 
\end{definition}

\begin{remark}
Two remarks are in order. First, note that by the definition of $h^*$, 
$g_u:= h - \ipt{u}{\cdot}$ being lower bounded is equivalent to $u\in \dom h^*$. However, we prefer to use the former statement in Definition~\ref{def:aff_att} since it does not (explicitly) involve $h^*$. Second, by~\cite[Theorem~2.2.8]{Renegar_01}, the class of (standard, strongly, non-degenerate) self-concordant functions are indeed affine attaining. 
\end{remark}


As an important observation, $h$ being affine attaining is necessary for $\dom h^*$ to be open. 

\begin{prop} \label{prop:affine_attaining}
Let $h:\bbX\to\barbbR$ be proper, closed and convex. If $\dom h^*$ is open, then $h$ must be affine attaining. 
\end{prop}

\begin{proof}
If for some $u\in\bbX^*$, $g_u:= h - \ipt{u}{\cdot}$ is lower bounded, then 
we have $u\in\dom h^*$. Since $\dom h^*$ is open, we have $\dom h^* = \inter\dom h^*$ and hence $u\in \inter \dom h^*$. By~\cite[Fact 2.11]{Bauschke_97}, we know that $g_u$ is coercive. Since $g_u$ is additionally proper and closed, we know that it has a minimizer on $\bbX$. 
\end{proof}

The following proposition provides an equivalent characterization of $\dom h^*$ being open. 

\begin{prop}[{\cite[Corollary~13.3.4(c)]{Rock_70}}] \label{prop:recession_equiv}
Let $h:\bbX\to\barbbR$ be a proper, closed and convex function. Then $\dom h^*$ is open if and only if for all $u\in\bbX^*$ such that $g_u:= h - \ipt{u}{\cdot}$ is lower bounded, 
$r_h(v) > \ipt{u}{v}$ for all $v\ne 0$. 
\end{prop}

Based on Proposition~\ref{prop:recession_equiv}, we present our first sufficient condition for $\dom h^*$ to be open.

\begin{lemma} \label{lem:assump_h_dual_attain}
Let $h:\bbX\to\barbbR$ be proper, closed and convex.  If $h$ is strictly convex (on its domain), then $\dom h^*$ is open if and only if $h$ is affine attaining. 
In particular, if $h$ satisfies Assumption~\ref{assum:h}, then $\dom h^*$ is open if and only if $h$ is affine attaining. 
\end{lemma}

\begin{proof}
The ``only if'' direction follows from Proposition~\ref{prop:affine_attaining}, and we only focus on the ``if'' direction. 
Let $u\in\bbX^*$ satisfy that $g_u:=h - \ipt{u}{\cdot}$ is lower bounded. Since $h$ is affine attaining, $g_u$ has a minimizer on $\bbX$, which we denote by $x^*\in \dom h$. By the optimality condition, we have $u\in \partial h(x^*)$. 
Since $h$ is strictly convex, we have 
\begin{equation}
h(x^*+ v) - h(x^*) > \ipt{u}{v}.  \label{eq:h_x^*+v}
\end{equation}
Since $x^*\in \dom h$, using the definition of $r_h$ in~\eqref{eq:def_r_h}, 
we know that for all $v\ne 0$, $r_h(v) > \ipt{u}{v}$. 
Using Proposition~\ref{prop:recession_equiv}, we prove the first part.
Now, suppose that $h$ satisfies Assumption~\ref{assum:h}. If $\dom h$ is singleton, then $h^*$ is linear and $\dom h^*=\bbX^*$, which is clearly open; otherwise $h$ is strictly convex on $\dom h$ (cf.\ Lemma~\ref{lem:h^*_diff}), and using the first part, we complete the proof. 
\end{proof}

Our next sufficient condition is based on the notion of Legendre functions (cf.\ Section~\ref{sec:suff_assump_h}). 

\begin{lemma} \label{lem:legendre_dual_attain}
If $h:\bbX\to\barbbR$  is Legendre,  then $\dom h^*$ is open if and only if $h$  is affine attaining. 
\end{lemma}

\begin{proof}
The ``only if'' direction follows from Proposition~\ref{prop:affine_attaining}, and we only focus on the ``if'' direction. 
For any $u\in\dom h^*$, 
since $g_u:=h - \ipt{u}{\cdot}$ is lower bounded and $h$ is affine attaining, $g_u$ has a minimizer on $\bbX$, which we denote by $x^*\in \dom h$. Since $h$  is Legendre, we actually have $x^*\in \inter\dom h$ and $u = \nabla h(x^*)$ (cf.~\cite[Theorem~26.1]{Rock_70}). By Lemma~\ref{lem:legendre}, we know that $u\in \inter\dom h^*$. This shows that $\dom h^*\subseteq \inter\dom h^*$, and we complete the proof. 
\end{proof}

{\bf Illustrating Examples.} Let us illustrate our results above using the examples in Example~\ref{eg:sep_h}, all of which are Legendre and satisfy Assumption~\ref{assum:h}. However, not all examples are affine attaining. 

\begin{itemize}[leftmargin=3ex,topsep=0pt]
\item $h_1(x):= \sum_{i=1}^n -\ln x_i$ (for $x>0$) is affine attaining. Indeed,  $h_1 - \ipt{u}{\cdot}$ is lower bounded if and only if $u<0$, in which case it has the unique minimizer  $x^* = [-1/u_i]_{i=1}^m$. By Lemma~\ref{lem:legendre_dual_attain}, we know that $\dom h_1^*$ is open, which is corroborated by the fact that 
$h_1^*(u) = \sum_{i=1}^n -\ln(-u_i) -1$. 
\item $h_2(x):= \sum_{i=1}^n x_i \ln x_i - x_i$ ($x\ge 0$) is affine attaining. Indeed,  $h_2 - \ipt{u}{\cdot}$ is lower bounded for all  $u\in\bbR^n$, in which case it has the unique minimizer  $x^* = [\exp(u_i)]_{i=1}^m$. By Lemma~\ref{lem:legendre_dual_attain}, we know that $\dom h_2^*$ is open, which is corroborated by the fact that 
$h_2^*(u) = \sum_{i=1}^n \exp(u_i)$. 
\item $h_3(x):= \sum_{i=1}^n \exp(x_i)$ ($x\in\bbR^n$) is not affine attaining, since  $h_3  = h_3- \ipt{0}{\cdot}$ is lower bounded but has no minimizer on $\bbX$. By Proposition~\ref{prop:affine_attaining}, we know that $\dom h^*_3$ is not open, which can also be seen from the facts that $h^*_3 = h_2$ and $\dom h_2 = \bbR_+^n$.  %
\end{itemize}

Let us conclude this section by the following result. 

\begin{lemma} \label{lem:sum_dom_open}
Let $h_1,h_2:\bbX\to\barbbR$ be proper, closed and convex functions such that $\ri \dom h_1\cap \ri \dom h_2\ne \emptyset$, and 
let $h:= h_1+h_2$. If $\dom h_1^*$ is open, then $\dom h^*$  is open. 
\end{lemma}

\begin{proof}
Since $\ri \dom h_1\cap \ri \dom h_2\ne \emptyset$, by~\cite[Theorem~16.4]{Rock_70}, we have $h^* = (h_1+h_2)^*= h_1^*\,\square\,h_2^*$, 
and hence from~\cite[pp.~34]{Rock_70}, we know that 
\begin{equation}
\dom h^* =   \dom (h_1^*\,\square\,h_2^*) =\dom h_1^*+\dom h_2^*. 
\end{equation}
Since $\dom h_1^*$ is open, $\dom h^*$ is open. 
\end{proof}

As a corollary, let $h_1$ be given in Lemma~\ref{lem:sum_dom_open}, 
and  $\calC$ be a nonempty, closed and convex set such that $\ri \dom h_1\cap \ri \calC\ne \emptyset$. If $\dom h_1^*$ is open, then $\dom (h_1+\iota_\calC)^*$  is open. 

\section{Affine Invariance of Algorithm~\ref{algo:DA} and Its Convergence Rate   Analysis } \label{sec:aff_inv} %

In this section, we discuss the affine invariance of Algorithm~\ref{algo:DA}  and its convergence rate  analysis in Theorem~\ref{thm:pdgap}. We start by formally introducing the notion of affine invariance. Then we show that Assumptions~\ref{assum:h} and~\ref{assum:Q} are still satisfied under the affinely re-parameterized problem, and Algorithm~\ref{algo:DA} is affine invariant. Finally, we show that 
if $\calU$ is solid and $\normt{\cdot}_{\bbX}$ is induced by $\calU$ in a certain way, the convergence rate  analysis of Algorithm~\ref{algo:DA} in Theorem~\ref{thm:pdgap} is also affine invariant. As a remark, 
although the discussions in this section focus on Algorithm~\ref{algo:DA}, the same reasoning can also be used to analyze the affine invariance of Algorithm~\ref{algo:MDA} and its convergence rate   analyses in Theorems~\ref{thm:analysis_algo2} and~\ref{thm:conv_rate_algo2}. 

\subsection{\bf Introduction to Affine Invariance} \label{sec:def_aff}

Given an optimization problem 
\begin{equation}
{\min}_{u\in\bbU}\, F(u), \label{eq:affine_orig}
\end{equation}
where $F$ is a proper and closed function,  let us define $\barcalA:= \aff (\dom F)$ and $\barcalL:= \lin \barcalA$. 
Consider the following {\em affine re-parameterization} of~\eqref{eq:affine_orig}: 
\begin{equation}
{\min}_{w\in\bbW}\, F(\rvM w + b). \label{eq:affine_reparam}
\end{equation}
Here $\rvM: \bbW\to \calL$ is a linear operator, where $\bbW$ and $\calL$ are (finite-dimensional) vector spaces such that $\barcalL\subseteq \calL\subseteq \bbU$, and $b\in\dom F$. 
An optimization algorithm $\scA$ is called {\em affine-invariant}, if the sequences of iterates $\{u^k\}_{k\ge 0}$ and $\{w^k\}_{k\ge 0}$ produced by $\scA$ when applied to~\eqref{eq:affine_orig} and~\eqref{eq:affine_reparam}, respectively, are related through the affine transformation $w\mapsto \rvM w + b$. Precisely, 
if $u^0 = \rvM w^0 + b$ (where $x_0$ and $w_0$ are the  starting points in $\scA$), then $u^k = \rvM w^k + b$ for all $k\ge 1$. In addition, if $\scA$ is {affine-invariant}, then a convergence rate analysis of $\scA$ is  {\em affine-invariant} if all the quantities appearing in the convergence rate remain unchanged after the affine re-parameterization  in~\eqref{eq:affine_reparam}. 

\subsection{\bf Affine Invariance of DA for Solving~\ref{eq:P}}

Following 
Section~\ref{sec:def_aff}, in the problem~\ref{eq:P},  define $\barcalA = \aff \dom h$ and $\barcalL:= \lin \barcalA$. Using the affine transformation $w\mapsto\rvM w + b$ described above, 
 where  $\rvM: \bbW\to \calL$ is a linear operator,   $\calL$ is some linear subspace such that $\barcalL\subseteq \calL\subseteq \bbX$ and $b\in\dom h$,  
 the affine re-parameterization of~\ref{eq:P} reads
\begin{align}
\begin{split}
&{\min}_{w\in\bbW}\; \tilf(\tilde\rvA w) + \tilh(w),  \label{eq:affine_reparam_P}\\
\where \tilde\rvA := \rvA \rvM, \quad &\tilf(z) := f(z+\rvA b), \quad\andd  \quad \tilh(w):= h(\rvM w + b). 
\end{split}\tag*{$(\rmP_w)$}
\end{align}
To 
state the affine invariance property of the DA method in Algorithm~\ref{algo:DA}, we restrict the class of linear operators $\rvM$ to the class of {\em linear bijections} from $\bbW$ to $\bbX$ (in particular, $\calL=\bbX$ and $\bbW$ has the same dimension as $\bbX$). 
The purpose of such a restriction is to ensure that if $h$ and $f$ in~\ref{eq:P} satisfy Assumptions~\ref{assum:h} and~\ref{assum:Q},  then $\tilh$ and $\tilf$ in~\ref{eq:affine_reparam_P} also satisfy these two assumptions. 

\begin{lemma}\label{lem:well_defined_w}
In~\ref{eq:affine_reparam_P},  let $\rvM:\bbW\to \bbX$ be a linear bijection and $b\in\dom h$. If $h$ and $f$ in~\ref{eq:P} satisfy Assumptions~\ref{assum:h} and~\ref{assum:Q}, so do $\tilh$ and $\tilf$ in~\ref{eq:affine_reparam_P}, 
  and hence Algorithm~\ref{algo:DA} is well-defined 
  on~\ref{eq:affine_reparam_P}. 
  \end{lemma}

  \begin{proof}
  Note that since $\rvM:\bbW\to \bbX$ is  a linear bijection, we have 
\begin{align}
\tilf^*(y) &= {\sup}_{z\in\bbY^*}\, \ipt{y}{z} - \tilf(z)\nn\\
& = {\sup}_{z\in\bbY^*}\, \ipt{y}{z} - f(z+\rvA b) = {\sup}_{z'\in\bbY^*}\, \ipt{y}{z'} - f(z')  - \ipt{y}{\rvA b} = f^*(y) - \ipt{y}{\rvA b}, \label{eq:def_tilf*} \\
\tilh^*(v) &= {\sup}_{w\in\bbW}\, \ipt{v}{w} - \tilh(w)\nn\\
& = {\sup}_{w\in\bbW}\, \ipt{v}{w} - h(\rvM w + b)\nn\\
 &\hspace{0ex} = {\sup}_{x\in\bbX}\, \ipt{v}{\rvM^{-1}(x-b)} - h(x) = h^*\big(\big(\rvM^{-1}\big)^* v\big) - \ipt{v}{\rvM^{-1}b}. \label{eq:def_tilh*}
\end{align}
 As a result, we have 
\begin{equation}
\dom \tilf^* = \dom f^* \quad\andd \quad \dom \tilh^* = ((\rvM^{-1})^*)^{-1} \dom h^* = \rvM^* \dom h^*. 
\end{equation}
Denote the counterparts of $\calQ$ and  $\calU$ in~\ref{eq:affine_reparam_P} by $\tilcalQ$ and $\tilcalU$, respectively, i.e., 
 \begin{align}
 \tilcalQ:= \cl \dom \tilf^*  = \cl \dom f^*  = \calQ \quad\andd \quad  \tilcalU := -\tilde\rvA^*(\tilcalQ) = -\rvM^*\rvA^*(\calQ)=  \rvM^*(\calU).  \label{eq:def_tilQ}
 \end{align}
 Since  $\rvM^*:\bbX^*\to \bbW^*$ is a linear bijection, we have $\inter \dom \tilh^*= \inter(\rvM^* \dom h^*) = \rvM^* (\inter\dom h^*)$. If $h$ and $f$ satisfy Assumption~\ref{assum:Q}, we have $\calU \subseteq \inter \dom h^*$ and hence $\tilcalU =  \rvM^*(\calU)\subseteq \rvM^* (\inter\dom h^*) = \inter \dom \tilh^*,$ 
 which verifies Assumption~\ref{assum:Q} for $\tilh$ and $\tilf$. To verify Assumption~\ref{assum:h} for $\tilh$, first note that for any nonempty, convex and compact set $\calS'\subseteq\dom \tilh$, the set $\calS:= \rvM(\calS')+b\subseteq \dom h$ is nonempty, convex and compact, and since $h$ satisfies Assumption~\ref{assum:h}, we have  $\mu_{\calS} > 0$. Now, since $\rvM$ is bijective, $\calS$ is singleton if and only if $\calS'$ is, in which case $\mu_{\calS'} := 1 > 0$. Otherwise, by choosing the norm $\normt{\cdot}_{\bbW}$ such that  $\normt{w}_{\bbW} := \normt{\rvM w}_{\bbX}$ for all $w\in\bbW$,  we have
 \begin{align*}
\begin{split}
\mu_{\calS'} &:= \inf\left\{\frac{\lambda \tilh( w) + (1-\lambda)\tilh(w')-\tilh((1-\lambda)w'+\lambda w)  }{(1/2) \lambda(1-\lambda)\normt{w'-w}_{\bbW}^2 }: \, w', w\in\calS',\, w'\ne w,\, \lambda\in (0,1)\right\} \\
& = \inf\Bigg\{\frac{\lambda h(\rvM w+b) + (1-\lambda)h(\rvM w'+b)-h((1-\lambda)\rvM w'+\lambda \rvM w+b)  }{(1/2) \lambda(1-\lambda)\normt{\rvM w'-\rvM w}_{\bbX}^2 }:\\
 &\hspace{16em} w', w\in\rvM^{-1}(\calS - b),\, w'\ne w,\, \lambda\in (0,1)\Bigg\}
 \end{split}\\
 & = \inf\Bigg\{\frac{\lambda h(x) + (1-\lambda)h(x')-h((1-\lambda)x'+\lambda x)  }{(1/2) \lambda(1-\lambda)\normt{x'-x}_{\bbX}^2 }: x', x\in\calS,\, x'\ne x,\, \lambda\in (0,1)\Bigg\}\\
 & = \mu_{\calS} > 0. 
\end{align*}
Since the positivity of $\mu_{\calS'}$ is independent of the choice of $\normt{\cdot}_\bbW$ (cf.~Remark~\ref{rmk:choice_norm}),  we know that $\tilh$ satisfies Assumption~\ref{assum:h} under any choice of $\normt{\cdot}_\bbW$. 
  \end{proof}
  
  Once we ensure that Algorithm~\ref{algo:DA} is well-defined when applied to~\ref{eq:affine_reparam_P} (cf.~Lemma~\ref{lem:well_defined_w}), using induction, we can easily show that it is affine-invariant, which is formally stated below.  

  
  \begin{theorem}\label{thm:aff_inv_DA}
  Let $\rvM:\bbW\to \bbX$ be a linear bijection and $b\in\dom h$, and Assumptions~\ref{assum:h} and~\ref{assum:Q} hold.  Apply Algorithm~\ref{algo:DA}  to~\ref{eq:P} and~\ref{eq:affine_reparam_P} with pre-starting points $x^{-1}\in\bbX$ and $w^{-1}\in\bbW$, respectively, and 
  denote the iterates generated by Algorithm~\ref{algo:DA} on~\ref{eq:P} and~\ref{eq:affine_reparam_P} by $\{x^k\}_{k\ge 0}$ and $\{w^k\}_{k\ge 0}$, respectively. In addition, for all $k\ge -1$, let $g^k$ 
  be chosen in the same way in Algorithm~\ref{algo:DA} when applied to~\ref{eq:P} and~\ref{eq:affine_reparam_P}. 
  Then $x^k = \rvM w^k + b$ for all $k\ge 0$ provided that $x^{-1} = \rvM w^{-1} + b$.
  
  \end{theorem}

%
%

\subsection{\bf Affine Invariance of the Convergence Rate Analysis of DA} \label{sec:aff_inv_analysis}

 As introduced in Section~\ref{sec:def_aff}, to analyze the affine invariance of the convergence rate  analysis of Algorithm~\ref{algo:DA}  in Theorem~\ref{thm:pdgap}, we only to focus on $\diam_{\normt{\cdot}_*}(\calU)$ and $\mu_{\barcalS}$ that appear in the convergence rate in~\eqref{eq:comp_guarantee_DA}. 
 Since the definitions of $\diam_{\normt{\cdot}_*}(\calU)$ and $\mu_{\barcalS}$ (cf.~\eqref{eq:def_diamU} and~\eqref{eq:def_muS}) involve the pair of norms $\normt{\cdot}_{\bbX}$ and $\normt{\cdot}_{\bbX,*}$, to make 
both quantities  affine invariant,  we need to choose a suitable norm $\normt{\cdot}_{\bbX,*}$ (or equivalently, $\normt{\cdot}_{\bbX}$) so that it ``adapts to'' the the affine re-parameterization. To that end, assume that $\calU$ is solid (i.e., $\inter\calU\ne \emptyset$). Since $\calU\ne\emptyset$ is convex and compact, we know  that $\calU - \calU$  is solid, compact, convex and symmetric around the origin. As a result, the gauge function of $\calU - \calU$  (cf.~\cite[pp.~28]{Rock_70}), namely
\begin{equation}
\gamma_{\calU - \calU} (u):= \inf\{\lambda>0: u/\lambda \in \calU - \calU\},  \label{eq:U_gauge}
\end{equation} 
is indeed a {\em norm} on $\bbX^*$. For convenience, define $\normt{\cdot}_\calU := \gamma_{\calU - \calU}$, 
and for the affine invariance analysis of $\diam_{\normt{\cdot}_*}(\calU)$ and $\mu_{\barcalS}$ in this subsection, we shall choose $\normt{\cdot}_{\bbX,*}  :=  \normt{\cdot}_\calU$. As a result, we have 
\begin{equation}
\normt{x}_\bbX = \normt{x}_{\calU,*}:= {\max}_{\normt{u}_\calU\le 1}\;  \ipt{u}{x} = {\max}_{{u}\in \calU-\calU}\;  \ipt{u}{x}, \quad \forall\,x\in\bbX.   \label{eq:norm_X} 
\end{equation}
Note that both $\normt{\cdot}_{\calU}$  and $\normt{\cdot}_{\calU,*}$ are induced by $\calU$, which only depends on $f$ and $\rvA$ and hence is {\em intrinsic} to the problem in~\ref{eq:P}. 
Consequently, $\diam_{\normt{\cdot}_*}(\calU)$ now becomes $\diam_{\normt{\cdot}_{\calU}}(\calU)$. As shown in the lemma below, 
we always have $\diam_{\normt{\cdot}_{\calU}}(\calU)=1$.  

\begin{lemma} \label{lem:diamU}
If $\inter\calU\ne \emptyset$, then $\diam_{\normt{\cdot}_{\calU}}(\calU):= \max_{u,u'\in\calU}\, \normt{u-u'}_\calU = 1$. 
\end{lemma} 

\begin{proof}
See Appendix~\ref{app:proof_diamU}. 
\end{proof}

Now, let us turn our attention to the re-parameterized problem~\ref{eq:affine_reparam_P}, where $\rvM:\bbW\to \bbX$ is a linear bijection and $b\in\dom h$. If $\calU$ is solid, then  its counterpart in~\ref{eq:affine_reparam_P}, i.e., $\tilcalU =   \rvM^*(\calU)$, is solid as well. Therefore, following~\eqref{eq:norm_X}, we define $\normt{\cdot}_{\bbW,*}  :=  \normt{\cdot}_{\tilcalU} = \gamma_{\tilcalU-\tilcalU}$ and 
\begin{equation}
\normt{w}_\bbW := {\max}_{{v}\in \tilcalU-\tilcalU}\;  \ipt{v}{w} = {\max}_{{u}\in \calU-\calU}\;  \ipt{\rvM^* u}{w} = \normt{\rvM w}_\bbX, \quad \forall\,w\in\bbW. \label{eq:norm_W}
\end{equation}
To show $\diam_{\normt{\cdot}_\calU}(\calU)$ and $\mu_{\barcalS}$ are affine-invariant, we simply need to  show that they are equal to their counterparts in~\ref{eq:affine_reparam_P}, i.e., $\diam_{\normt{\cdot}_\calU}(\calU) = \diam_{\normt{\cdot}_{\tilcalU}}(\tilcalU)$ and $\mu_{\barcalS} = \mu_{\tilcalS}$, where $\tilcalS := \conv(\nabla \tilh^*(\tilcalU))$ and 
 \begin{align}
\mu_{\tilcalS}: = \inf\left\{\frac{\lambda \tilh( w) + (1-\lambda)\tilh(w')-\tilh((1-\lambda)w'+\lambda w)  }{(1/2) \lambda(1-\lambda)\normt{w'-w}_{\bbW}^2 }: \, w', w\in\tilcalS,\, w'\ne w,\, \lambda\in (0,1)\right\} \label{eq:def_mu_tilS} 
\end{align} 
if $\tilcalS$ is non-singleton, and $\mu_{\tilcalS}:= 1$ otherwise (cf.\ Assumption~\ref{assum:h}). 


  
  \begin{theorem} \label{thm:aff_inv_DA_analysis}
  Let $\calU$ be solid, $\rvM:\bbW\to \bbX$ be a linear bijection and $b\in\dom h$.
  If $\normt{\cdot}_{\bbX}$ and $\normt{\cdot}_{\bbW}$ are induced by $\calU$ and $\tilcalU$ as in~\eqref{eq:norm_X} and~\eqref{eq:norm_W}, respectively,  
  then 
  $\diam_{\normt{\cdot}_\calU}(\calU) = \diam_{\normt{\cdot}_{\tilcalU}}(\tilcalU)=1$  and $\mu_{\barcalS} = \mu_{\tilcalS}$. 
  \end{theorem}

\begin{proof}\renewcommand{\qedsymbol}{}
By Lemma~\ref{lem:diamU}, we clearly see that $\diam_{\normt{\cdot}_\calU}(\calU) = \diam_{\normt{\cdot}_{\tilcalU}}(\tilcalU)=1$, and hence we only need to show that $\mu_{\barcalS} = \mu_{\tilcalS}$. 
From~\eqref{eq:def_tilh*} and~\eqref{eq:def_tilQ}, we have 
\begin{align}
\tilcalS = \conv(\nabla \tilh^*(\tilcalU)) &=\conv\big(\rvM^{-1}\big(\nabla h^* \big(\big(\rvM^{-1}\big)^*\rvM^*(\calU)\big) - b\big)\big)\\
& = \conv(\rvM^{-1}(\nabla h^* (\calU) - b))\\
& = \rvM^{-1}(\conv(\nabla h^* (\calU) - b))\\
& = \rvM^{-1}(\conv(\nabla h^* (\calU)) - b) =\rvM^{-1}(\bar\calS - b). \label{eq:def_tilS} 
\end{align}
Now, note that by~\eqref{eq:def_tilS}, $\tilcalS$ is a singleton if and only if $\bar\calS$ is, in which case $\mu_{\barcalS} = \mu_{\tilcalS}=1$. For non-singleton $\tilcalS$, by the definition of $\tilh$ in~\ref{eq:affine_reparam_P},~\eqref{eq:def_mu_tilS},~\eqref{eq:def_tilS} and~\eqref{eq:norm_W}, we have 
\begin{align}
\begin{split}
\mu_{\tilcalS}& = \inf\Bigg\{\frac{\lambda h(\rvM w+b) + (1-\lambda)h(\rvM w'+b)-h((1-\lambda)\rvM w'+\lambda \rvM w+b)  }{(1/2) \lambda(1-\lambda)\normt{\rvM w'-\rvM w}_{\bbX}^2 }:\\
 &\hspace{16em} w', w\in\rvM^{-1}(\barcalS - b),\, w'\ne w,\, \lambda\in (0,1)\Bigg\}
 \end{split}\nn\\
 & = \inf\Bigg\{\frac{\lambda h(x) + (1-\lambda)h(x')-h((1-\lambda)x'+\lambda x)  }{(1/2) \lambda(1-\lambda)\normt{x'-x}_{\bbX}^2 }: x', x\in\barcalS,\, x'\ne x,\, \lambda\in (0,1)\Bigg\} = \mu_{\barcalS}. \tag*{$\square$}
\end{align}
\end{proof}


\begin{remark}

Note that the solidity of $\calU$ 
is not needed for the DA method in Algorithm~\ref{algo:DA} to be affine invariant (cf.~Theorem~\ref{thm:aff_inv_DA}),  
but is needed in~Theorem~\ref{thm:aff_inv_DA_analysis}  to show the affine invariance of the convergence rate analysis in Theorem~\ref{thm:pdgap}. Specifically, we need the solidity of $\calU$ 
to ensure that $\normt{\cdot}_{\calU}  := \gamma_{\calU-\calU}$ 
is indeed a norm, and also that $\tilcalU$ is solid under the linear bijection $\rvM:\bbW\to \bbX$, which in turn ensures that $\normt{\cdot}_{\tilcalU}  := \gamma_{\tilcalU-\tilcalU}$ is a norm. 
 That said, there may exist some other convergence rate analyses of the DA method that are affine invariant without requiring $\calU$ to be solid, and we leave this to future work. 

\end{remark}


%

\section{Relaxing the Globally Convex and Lipschitz Assumptions of $f$}


So far, all of our results have been obtained based on the globally convex and Lipschitz assumptions of $f$. 
 However, one easily observes that the optimal objective value of~\ref{eq:P}, as well as the optimal solution(s) of~\ref{eq:P} (if any),  only depends on 
the part of $f$ that is defined on $\calC:=\rvA(\dom h)$, which is a nonempty and convex set (but may not be closed). 
%
In view of this, the globally Lipschitz assumption of $f$ may seem unnecessarily restrictive, 
and the same observation also applies to the globally convex assumption of~$f$. 
(In fact, when the prox-function $h$ is strongly convex, 
the classical analysis of the DA method indeed only requires $f$ to be convex and Lipschitz on $\calC$ -- see e.g.,~\cite{Nest_09,Grigas_16}.)

In this section, we will relax the globally convex and Lipschitz  assumptions of $f$, and instead focus on the setting where $f$ is only convex and $L$-Lipschitz on 
$\calC$.  
 As we shall see, in this case, we can obtain a convex and globally $L$-Lipschitz extension of $f$, denoted by $F_L$, by leveraging the notion of Pasch-Hausdorff (PH) envelope (cf.\ Proposition~\ref{lem:Lipschitz_extension}). This allows us to replace $f$ with $F_L$ in~\ref{eq:P}, which results in an  equivalent problem of~\ref{eq:P} that 
 satisfies our original assumptions on~\ref{eq:P} listed in Section~\ref{sec:intro}. 
 We show that we can obtain a subgradient of $F_L$ at any $z\in \calC$ if given access to $\partial f(z)$ and $\calN_\calC(z)$, where $$\calN_\calC(z):=\{y\in\bbY:\ipt{y}{z'-z}\le 0, \,\forall\,z'\in\calC\}$$ denotes the normal cone of $\calC$ at $z$ (cf.\ Proposition~\ref{prop: partial_FL}). In addition, we provide ways to obtain $F_L^*$ from $f^*$ and $\sigma_\calC$ (i.e., the support function of $\calC$), as well as obtain $\dom F_L^*$ from $\dom f^*$ and $\dom \sigma_\calC$. 
 
As a passing remark, note that the discussions in this section are solely on the convex analytic properties of  $f$ and its extension $F_L$. Due to this, 
they are not only relevant in developing and analyzing Algorithms~\ref{algo:DA} and~\ref{algo:MDA}, but  any 
({feasible}) first-order method that {requires  $f$ in the objective to be globally convex and Lipschitz}. 

Before our discussions, we provide a simple example to illustrate the setting above.  
\begin{example} \label{eg:nonLips}
Consider the following optimization problem: 
\begin{equation}
{\min}_{x\in\bbR^n}\; \textstyle -\sum_{i=1}^m\; \ln(a_i^\top x) + \max_{i\in[m]}\, a_i^\top x + \sum_{i=1}^n x_i \ln x_i - x_i + \iota_\calX(x),  \label{eq:min_entropy}
\end{equation}
where $a_i\in\bbR_{++}^n$ for $i\in[m]$ and $\calX:=\bbR_+^n+e$. Putting~\eqref{eq:min_entropy} in the form of~\ref{eq:P}, we have
\[
\textstyle f:z\mapsto -\sum_{i=1}^m  \ln z_i + \max_{i\in[m]}\, z_i, \quad \rvA: x\mapsto Ax 
\quad\andd\quad  h: x\mapsto  \sum_{i=1}^n x_i \ln x_i - x_i + \iota_\calX(x),
\]
where $A := [a_1 \, \cdots\, a_m]^\top\in \bbR_{++}^{m\times n}$. 
 Clearly, $f$ 
is not globally Lipschitz on $\bbR^m$, but is Lipschitz on $\calC =\rvA(\calX) = \cone\{A_j\}_{j=1}^{n}+Ae $, where  $A_j$ denotes the $j$-th column of $A$, for $j\in[n]$. 
\end{example}

Now, let us present the main results in this section. We start by introducing the PH envelope. 


\begin{definition}[{Infimal convolution and the PH envelope~\cite[Section~12]{Bauschke_11}}] \label{def:PH}
Let $\bbU:=(\bbR^d, \normt{\cdot})$ be a normed space.  
Given two proper 
functions $\phi,\omega:\bbU\to\barbbR$, define their infimal convolution $\phi\,\square\, \omega : \bbU\to\barbbR\cup\{-\infty\}$ as 
\begin{equation}
(\phi\,\square\, \omega) (u) := \textstyle \inf_{u'\in \bbU}\;  \phi(u') + \omega(u-u'), \quad \forall\, u\in \bbU. 
\end{equation}
In particular, if $\omega = \gamma\normt{\cdot}$ for some $\gamma>0$, then $
f\,\square\, \gamma\normt{\cdot}$ is called the $\gamma$-PH 
envelope of $f$.
\end{definition}

Define $f_\calC := f+\iota_\calC$, which is proper, convex and $L$-Lipschitz on $\calC$. 
Based on Definition~\ref{def:PH},   let $F_L:=f_\calC\,\square\,L\normt{\cdot}_{*}$ be the $L$-PH envelope of $f_\calC$, i.e.,  
\begin{equation}
F_L(z) := \textstyle \inf_{z'\in \bbY^*}\;  f_\calC(z') + L\normt{z-z'}_*, \quad \forall\, z\in \bbY^*. \label{eq:PH_env}
\end{equation}


The following proposition shows that $F_L$ is indeed a globally convex and Lipschitz extension of $f$. The proof is rather simple and can be found in e.g.,~\cite[Section~12.3]{Bauschke_11}. For completeness, we provide its proof in Appendix~\ref{app:proof_Lipschitz_extension}. 

\begin{prop} \label{lem:Lipschitz_extension}
If $f$ is convex and $L$-Lipschitz on $\calC$, then $F_L$ is convex and $L$-Lipschitz on $\bbY^*$, and $F_L = f$ on $\calC$. 
In particular, $f_\calC = F_L + \iota_\calC$. 
\end{prop}


Based on Proposition~\ref{lem:Lipschitz_extension}, if $f$ is only convex and $L$-Lipschitz on $\calC$, 
then we can instead solve the following equivalent problem: 
\begin{equation}
{\min}_{x\in\bbX}\; F_L(Ax) + h(x), \tag*{$(\rm P_e)$} \label{eq:P_e}
\end{equation}
where $F_L$ indeed satisfies our original assumption about $f$ in Section~\ref{sec:intro}. 

One natural question that one may have about solving~\ref{eq:P_e} is that how to compute a subgradient of $F_L$ at given $z\in\bbY^*$.  
For $z\not\in\calC,$ this requires solving the optimization problem in~\eqref{eq:PH_env} in general. However, for $z\in\calC$, as we show in the next proposition, if we are given access to $\partial f(z)$ and $\calN_\calC(z)$, then we can obtain a subgradient $g\in \partial F_L(z)$ without solving the optimization problem in~\eqref{eq:PH_env}. This result is particularly relevant to most of the feasible first-order methods for solving~\ref{eq:P_e} (including both Algorithms~\ref{algo:DA} and~\ref{algo:MDA}), where the primal iterates $\{x_k\}_{k\ge 0}\subseteq\dom h$, and  subgradients of $F_L$ are computed at the iterates $\{\rvA x_k\}_{k\ge 0}\subseteq\calC$. 
%
%

\begin{prop} \label{prop: partial_FL}
Define $\calB_{\normt{\cdot}}(0,L):=\{y\in\bbY:\normt{y}\le L\}.$ 
For any $z\in\calC$, we have 
\begin{align}
&\; \big(\partial f(z) + \calN_\calC(z)\big) \cap\calB_{\normt{\cdot}}(0,L) \subseteq \partial f_\calC(z)\cap\calB_{\normt{\cdot}}(0,L)=\partial F_L(z)\ne \emptyset. \label{eq:fC_in_FL}
\end{align}
In addition, if $\ri\dom f \cap\ri\calC\ne \emptyset$, then the set inclusion  in~\eqref{eq:fC_in_FL} becomes equality. 
\end{prop}

\begin{proof}
The proof leverages simple and  basic convex analytic arguments --- see Appendix~\ref{app:proof_partial_FL}. 
\end{proof}

%

From Proposition~\ref{prop: partial_FL}, we know that for any $z\in\calC$, if there exist $g\in \partial f(z)$ and $g'\in\calN_\calC(z)$ such that $\normt{g+g'}\le L$, then $g+g'\in \partial F_L(z)$. Additionally, if  $\ri\dom f \cap\ri\calC\ne \emptyset$, then the converse is also true, namely, there must exist $g\in \partial f(z)$ and $g'\in\calN_\calC(z)$ such that $\normt{g+g'}\le L$. 
We also remark that if $\ri\dom f \cap\ri\calC= \emptyset$, then the converse fail to hold. For example, consider $f(z_1,z_2) = \abst{z_1} - \sqrt{z_2}$ with $\ri\dom f = \bbR\times \bbR_{++}$, and $\calC = \bbR\times\{0\}$. In this case, $\ri\dom f\cap\ri\calC = \emptyset$, $f_\calC = \iota_\calC$ and $F_L\equiv 0$ (with $L=0$). However, note that at any $z\in\calC$, $\partial f(z) = \emptyset$. 



Lastly, let us focus on $F_L^*$, which plays important roles in both Algorithms~\ref{algo:DA} and~\ref{algo:MDA}. 

\begin{prop} \label{prop:F_L*}
Let $\sigma_\calC:= \iota^*_\calC$ denotes the support function of $\calC$. 
We have 
\begin{equation}
F_L^*=   f_\calC^* + \iota_{\calB_{\normt{\cdot}}(0,L)}\quad\andd \quad \dom F_L^* = \dom f_\calC^*\cap \calB_{\normt{\cdot}}(0,L). 
\end{equation}
In addition, if $\ri\dom f \cap\ri\calC\ne \emptyset$, we have 
\begin{equation}
F_L^*=  ( f^*\,\square\, \sigma_\calC ) + \iota_{\calB_{\normt{\cdot}}(0,L)}\quad\andd \quad \dom F_L^* = (\dom f^* + \dom \sigma_\calC)\cap \calB_{\normt{\cdot}}(0,L). \label{eq:ri_nonempty}
\end{equation}
\end{prop}

\begin{proof}
This proof follows from the definitions of $F_L$ and $f_\calC$,~\cite[Theorem~16.4]{Rock_70} and~\cite[pp.~34]{Rock_70}.  
\end{proof}

\begin{remark} \label{rmk:F_L*}
Note that in some cases, $\dom F_L^*$ can be much easier to find compared to $F_L^*$ itself. To see this, consider Example~\ref{eg:nonLips}, where $f = f_1+f_2$ for $f_1:z\mapsto  -\sum_{i=1}^m  \ln z_i$ and $f_2:z\mapsto \max_{i\in[m]}\, z_i$. Clearly, $\ri \dom f_1\cap \ri \dom f_2\ne \emptyset$, and we have $f^* = f_1^* \,\square \, f_2^*$, where $f_1^*:y \mapsto -\sum_{i=1}^n \ln(-y_i) - n$ and $f_2^*:y \mapsto \iota_{\Delta_n}(y)$. Note that 
$\dom f^*$ can be easily determined as follows (cf.~\cite[pp.~34]{Rock_70}): 
\begin{equation}
\dom f^* = \dom f_1^* + \dom f_2^* = \bbR^{n}_{--} + \Delta_n = \{y\in\bbR^n: \, \textstyle \sum_{i\in\calI}\, y_i < 1, \;\; \forall\, \calI\subseteq [n], \, \calI\neq \emptyset\}. 
\end{equation}
However, note that $f^*$ has no closed-form expression, but for any $y\in \dom f^*$, we can compute $f^*(y)$ and $\nabla f^*(y)$ via a procedure that terminates in at most $n$ steps. (In fact, as $f$ satisfies Assumption~\ref{assum:h}, by Lemma~\ref{lem:h^*_diff}, we know that $f^*$ is differentiable on  $\dom f^*$.) 
In addition, recall that $\calC =\rvA(\calX) = Ae +\calK$ and $\calK:=\cone\{A_j\}_{j=1}^{n}$, where $A = [A_1\,\cdots\, A_n]$. Then we have $\sigma_\calC(y) = \ipt{y}{Ae} + \sigma_\calK(y) = \ipt{y}{Ae} + \iota_{\calK^\circ}(y)$ for $y\in \bbY$, where 
$$\calK^\circ := \{y\in \bbR^n: \ipt{y}{z}\le 0, \, \forall\, z\in \calK\} = \{y\in \bbY: A^\top y\le 0\}$$ 
denotes the polar cone of $\calK$. Since we clearly have $\dom \sigma_\calC = \calK^\circ$ and $\ri\dom f \cap\ri\calC\ne \emptyset$, by~\eqref{eq:ri_nonempty}, we have the following explicit description of $\dom F_L^*$, namely 
\begin{equation}
\dom F_L^* = \{y + y': \textstyle\sum_{i\in\calI}\, y_i < 1, \; \forall\, \calI\subseteq [n], \, \calI\neq \emptyset, \;\; A^\top y'\le 0,\;\; \normt{y+y'}\le L\}. 
\end{equation}
In contrast, note that given some $y\in\dom F_L^*$, it is difficult to compute $F_L^*(y)$  in general. In fact, according to~\eqref{eq:ri_nonempty}, this amounts to evaluating $( f^*\,\square\, \sigma_\calC )(y)$, which involves a non-trivial convex optimization problem that typically requires some  iterative algorithms to solve. 
\end{remark}


\section{Related Work} \label{sec:related_work}

We review two lines of work that analyze different subgradient methods  without leveraging the strong convexity of the prox-function.    
As we shall see, although these works also focus on convex nonsmooth optimization, the problem settings therein are quite different from ours. 

\subsection{Barrier subgradient method}

In an interesting work~\cite{Nest_11},  Nesterov proposed the ``barrier subgradient method'' (BSM) for solving the concave maximization problem 
\begin{equation}
{\max}_{x\in\calC}\, f(x), \label{eq:concave_max}
\end{equation}
where $\calC\subseteq\bbR^n$ is the intersection of two closed convex sets $\calP$ and $\calQ$ (namely, $\calC:=\calP\cap\calQ$), and is assumed to be {\em bounded} and ``simple''. 
In addition, $\calQ$ contains no line and  is endowed with a (non-degenerate) self-concordant  barrier (SCB), which is denoted by $F$ (cf.~\cite[Section~2.3]{Nest_94}). 
The objective function $f:\bbR^n\to \bbR\cup\{-\infty\}$ is 
concave on $\bbR^n$ and  subdifferentiable on $\calC^o:= \calP\cap\inter\calQ$, and satisfies that 
\begin{equation}
\normt{g}_{x}^*\le M, \qquad \forall\, g\in \partial f(x), \;\; \forall\, x\in \calC^o, \label{eq:bounded_subgrad_x}
\end{equation}
where $\normt{g}_{x}^*:= \ipt{\nabla^2 F(x)^{-1} g}{g}^{1/2}$  and  $M\ge 0$.\footnote{Note that here we abuse slightly the notion of subdifferentiability, and call a concave function $f$ subdifferentiable on $\calC^o$ if for any $x\in \calC^o$, $\partial f(x):= \{g\in\bbR^n: f(y)\le f(x)+\ipt{g}{y-x},\; \forall\, y\in\bbR^n\}\ne \emptyset$.} 
Typical examples of $f$ have the form $f:=\ln\circ\, \psi$, where  $\psi$ is a concave, positive and subdifferentiable function on $\inter\calQ$ (cf.~\cite[Lemma~5]{Nest_11}). The BSM developed for solving this problem (cf.~\cite[Eqn.~(14)]{Nest_11}) 
has the same structure as the original DA method in~\cite{Nest_09}, 
 but with the key difference that  $F$, which is a SCB on $\calQ$, plays the role of the prox-function. Indeed, the analysis of BSM 
 heavily leverages the  SCB  properties of $F$,  rather than the strong convexity of $F$ on $\calC$ (if any).   Under proper parameter choices, Nesterov showed that the BSM has a primal-dual convergence rate of $O(\ln(k)/\sqrt{k})$ for solving~\eqref{eq:concave_max}.   On the surface level, the BSM and Algorithm~\ref{algo:DA} share certain similarities, since they have similar structures and both of their analyses do not leverage the strong convexity of the prox-function. However, it is important to note that the problems these two methods intend to solve, namely~\eqref{eq:concave_max} and~\ref{eq:P}, are vastly different, in terms of both the objective function and the feasible region. 
  As a result, these two methods are highly different in several aspects, including parameters choices, convergence rate analyses  and rates of convergence. 
 
 

%


\subsection{Bregman subgradient method under the ``relative continuity'' condition}

In an insightful work~\cite[Section~3.2.2]{Gutman_23}, Gutman and Pe\~na considered 
the following convex nonsmooth 
problem: 
\begin{equation}
{\min}_{x\in\bbX}\; \psi(x) + \Psi(x), \label{eq:cvx_min}
\end{equation}
where both $\psi,\Psi:\bbX\to\barbbR$ are closed and convex functions such that $\emptyset\ne \dom \Psi\subseteq \dom \partial \psi$. Let $h:\bbX\to\barbbR$ be another closed and convex function that is differentiable on $\inter\dom h\ne \emptyset$. 
The authors proposed new analysis for the Bregman proximal subgradient method (BPSM), which   starts with any $x^0\in \dom \Psi \cap \inter\dom h$, and 
iterates:
\begin{equation}
x^{k+1} \in {\argmin}_{x\in\bbX}\;\; \ipt{g_k}{x} + \Psi(x) + t_k^{-1} D_h(x,x^k), \quad\forall\,k\ge 1,  
\label{eq:BPSM}
\end{equation}
where $g^k\in \partial \psi(x^k)$, $t_k>0$ and $D_h(x,x^k):=  h(x) - h(x^k) - \ipt{\nabla h(x^k)}{x - x^k}$. Their new analysis is based on the following condition: there exists $M\in(0,+\infty)$ 
such that for any $x\in \dom h\cap\dom \Psi$, $x'\in\inter\dom h\cap\dom\Psi$, $g\in\partial \psi(x')$ and $t>0$,  
\begin{equation}
M^2t^2/2 + t(\Psi(x) - \Psi(x') +\ipt{g}{x-x'}) + D_h(x,x')\ge 0.  \label{eq:relative_cont}
\end{equation}
Such a condition can be viewed as a generalization of the ``relative continuity'' 
condition initially proposed in~\cite{Lu_19b} and~\cite[Section~4.2]{Teboulle_18} 
when $\Psi = \iota_\calC$ for some closed convex set $\calC\ne \emptyset$. 
As shown in~\cite{Gutman_23}, under this condition, the BPSM in~\eqref{eq:BPSM} has a primal converge rate of $O(1/\sqrt{k})$ 
with proper choices of $\{t_k\}_{k\ge 0}$.  Note that  the condition in~\eqref{eq:relative_cont} do not explicitly require the strong convexity of the prox-function $h$ (as well as the Lipschitz continuity of $\psi$ and $\Psi$). That said, a primary situation where~\eqref{eq:relative_cont} holds  is indeed when both $\psi$ and $\Psi$ are Lipschitz on $\dom \Psi$ and $h$ is 1-strongly convex on $\dom \Psi$ (cf.~\cite{Gutman_23}). 
On the other hand, it is important to note that for the class of prox-functions $h$ considered in this work (see Example~\ref{eg:sep_h}), 
the condition in~\eqref{eq:relative_cont} {\em may not hold} on~\ref{eq:P} (for $\psi= f\circ \rvA$ and $\Psi = h$), even for the simple problem instance $\min_{x\in\bbR}\, x - \ln x$. In addition, note that the BPSM is a primal subgradient scheme, and in contrast, Algorithms~\ref{algo:DA} and~\ref{algo:MDA} 
are both dual subgradient schemes. 
As a result, the BPSM  has rather different structure and convergence rate analysis compared to those of Algorithms~\ref{algo:DA} and~\ref{algo:MDA}.

\section{Concluding Remarks: a Perspective From Frank-Wolfe}

In the seminal work~\cite[Section~3.3]{Grigas_16}, Grigas showed that 
when $f=\sigma_\calQ$ and $h$ is strongly convex (on its domain), 
the DA method in Algorithm~\ref{algo:DA}, when viewed from the dual, can be regarded as the Frank-Wolfe (FW) method~\cite{Frank_56} for solving~\ref{eq:D} with $h^*$ being a globally convex and smooth function and $f^* = \iota_\calQ$. 
Specifically, in the context of FW, the main sequence of iterates is $\{\bars_k\}_{k\ge 0}$ (cf.~\eqref{eq:def_bars}), and the step-sizes are given by $\{\alpha_k/\beta_{k+1}\}_{k\ge 0}$, which are commonly referred to as the ``open-loop'' step-sizes. 
Although the focus of this work is primarily on DA-type methods for solving the primal problem~\ref{eq:P}, such a dual viewpoint in terms of FW offers  
two insights on Algorithms~\ref{algo:DA} and~\ref{algo:MDA} in the setting of this work. 

\begin{itemize}[leftmargin = 0ex,label={},topsep=0pt,itemsep=5pt]
\item First, under Assumptions~\ref{assum:h} and~\ref{assum:Q}, 
 by Lemma~\ref{lem:h^*}, we know that $h^*$ is indeed $\mu_{\barcalS}^{-1}$-smooth on $\calU$ (which is  nonempty, convex and compact; cf.~Lemma~\ref{lem:compact}). Note that $\calU$ can be interpreted as the de facto feasible region of~\ref{eq:D} --- in particular, if we let $f^* = \iota_\calQ$, then~\ref{eq:D} can be written as $\min_{u\in \calU} \, h^*(u)$. The smoothness of $h^*$ on $\calU$ implies that
~\ref{eq:D} is a composite convex smooth optimization problem, and 
Algorithm~\ref{algo:DA} can be viewed as a (generalized) FW method  for solving~\ref{eq:D}. As a result, we can apply the analyses of the FW method with ``open-loop'' step-sizes (see e.g.,~\cite{Jaggi_13,Freund_16,Bach_15,Wirth_25}) to analyze Algorithm~\ref{algo:DA}, and obtain similar primal-dual convergence rate guarantees to those in Theorem~\ref{thm:pdgap}. Note that compared to this ``dual'' approach, the  approach in the proof of Theorem~\ref{thm:pdgap} proceeds on the primal side by directly making use of the strong convexity of $h$ on $\barcalS$ (cf.~Lemma~\ref{lem:compact}), and  avoids establishing the smoothness  of $h^*$  on $\calU$. 

\item Second, note that when Assumptions~\ref{assum:h} holds but Assumption~\ref{assum:Q} fails to hold, the function $h^*$ is no longer smooth on the ``feasible region'' $\calU$. Rather, it is smooth on any nonempty, convex and compact set inside its domain (cf.\ Lemma~\ref{lem:h^*}), which is assumed to be open (cf.\ Assumption~\ref{assum:open_dom}). 
Note that this setting is quite ``non-standard'' in the literature of the FW method, which typically assumes that $h^*$ is smooth on $\calU$. As such, our newly developed DA-type method in Algorithm~\ref{algo:MDA} 
can be viewed as a new (generalized) FW-type method for solving~\ref{eq:D} under this ``non-standard'' setting. 
  It should be mentioned that a recent line of works (see e.g.,~\cite{Dvu_23,Zhao_23}) have considered the setting where $h^*$ has 
  the {\em (generalized) non-degenerate self-concordance (NSC) property} (and hence may not be smooth on $\calU$), and developed new FW-type methods that have primal-dual converge rates of order $O(1/k)$. Note that the development and/or analyses of these methods crucially leverage many important properties of $h^*$ implied by the (generalized) NSC 
  property, e.g., certain curvature bound of $h^*$ 
(cf.~\cite{Nest_94,Sun_18}).
We emphasize that 
our 
model of $h^*$ above 
{\em strictly subsumes} the class of (generalized) NSC 
functions (which indeed 
have open domains and are %
smooth on any nonempty, convex and compact set inside their domains), and 
moreover, our assumptions on $h^*$ are {\em much easier to verify} than the 
(generalized) NSC property. 
Therefore, compared to the existing FW-type methods for optimizing the (generalized) NSC functions, Algorithm~\ref{algo:MDA} is developed for a more general and easier-to-verify setting, 
yet still has a primal-dual converge rate of order $O(1/k)$. In addition, 
compared to those FW-type methods, 
the analysis and computational guarantees of Algorithm~\ref{algo:MDA} (cf.\ Theorem~\ref{thm:analysis_algo2} and Remark~\ref{rmk:complexity}) are mostly geometric in nature, and in particular, they do not depend on any ``global'' curvature bound of $h^*$ that holds over its entire domain. 

\vspace{1ex}
To conclude this section, let us make a remark about Section~\ref{sec:aff_inv}. 
From the discussions above, we know that Algorithm~\ref{algo:DA}
can be viewed as the FW method for solving~\ref{eq:D}, which is a composite convex smooth optimization problem.  As such, it is tempting to think that the affine invariance of Algorithm~\ref{algo:DA} and its analysis directly follow from those of the FW method (see e.g.,~\cite{Jaggi_13,Wirth_25}). However, note that this is not the case, since the affine invariance analysis of the FW method focuses on the affine re-parameterization of the dual problem~\ref{eq:D}, whereas our affine invariance analysis of Algorithm~\ref{algo:DA} (cf.\ Section~\ref{sec:aff_inv}) focuses on the affine re-parameterization of the primal problem~\ref{eq:P}. Specifically, 
under the the affine re-parameterization of~\ref{eq:P}, 
we need to verify the validity of Assumptions~\ref{assum:h} and~\ref{assum:Q}, and also 
show  the invariance of the convergence rate  in Theorem~\ref{thm:pdgap} (under certain appropriate choice of  $\normt{\cdot}_{\bbX}$). 



\end{itemize}


\vspace{1em}
{\bf Acknowledgment.} The author sincerely thanks Louis Chen for helpful discussions which lead to Lemma~\ref{lem:assump_h_dual_attain}.

\appendix

\section{Proof of Lemma~\ref{lem:suff_h} } \label{app:proof_suff_h}

If $\calS$ is a singleton, then Lemma~\ref{lem:suff_h} holds trivially. Thus we focus on the case where $\calS$ is  not a singleton. Take any $x,y\in \calS$ such that $x\ne y$, and  
fix any sequence $\{\lambda_k\}_{k\ge 0}\subseteq (0,1)$ such that $\lambda_k\to 0$. Define $x^k := (1-\lambda_k) x + \lambda_k z$ and $y_k := (1-\lambda_k) y + \lambda_k z$, so that $x^k,y^k\in \calS_z^o$ for $k\ge 0$, and $x^k \to x$ and $y^k \to y$. We claim that $h(x^k)\to h(x)$. Indeed, 
we have
\begin{equation}
\limsup_{k\to+\infty}\, h(x^k)\lea \limsup_{k\to+\infty}\, (1-\lambda_k) h(x) + \lambda_k h(z) = h(x) \leb \liminf_{k\to+\infty}\, h(x^k), \label{eq:limsup_inf}
\end{equation}
where (a) and (b) follow from the convexity and closedness of $h$, respectively. By~\eqref{eq:limsup_inf}, we have $\lim_{k\to+\infty}\, h(x^k) = h(x)$. Similarly, we have $h(y^k)\to h(y)$. Now, fix any   $t\in(0,1)$. 
Since $x^k,y^k\in \calS_z^o$, by~\eqref{eq:kappa_S}, we know that 
\begin{equation}
h((1-t) x^k + ty^k)\le (1-t) h( x^k) + th(y^k) - ({\kappa_{\calS_z} (1-t)t}/{2})\normt{x^k-y^k}^2, 
\end{equation}
and hence by the closedness of $h$, we have 
\begin{align*}
h((1-t) x + ty)&\le \liminf_{k\to+\infty} \, h((1-t) x^k + ty^k)\\
&\le \liminf_{k\to+\infty} \, (1-t) h( x^k) + th(y^k) - ({\kappa_{\calS_z} (1-t)t}/{2})\normt{x^k-y^k}^2\\
&=  (1-t) h( x) + th(y) - ({\kappa_{\calS_z} (1-t)t}/{2})\normt{x-y}^2. 
\end{align*}
This completes the proof.  \qed

\section{Proof of Lemma~\ref{lem:sep} } \label{app:proof_sep}

Write $h(x) = \sum_{i=1}^n h_i(x_i)$, and note that 
\begin{enumerate}[label=\roman*),leftmargin=15pt,itemsep=5pt,parsep=0pt,topsep=1pt]
\item $h$ is very strictly convex and Legendre on $\bbR^n$ if and only if for each $i\in[n]$, $h_i$ is very strictly convex and Legendre on $\bbR$, and 
\item $\normt{\nabla^2 h(x^k)}\to +\infty$ for any $\{x^k\}_{k\ge 0}\subseteq \inter\dom h$ such that $x^k\to x\in \bdry \dom h$ if and only if for each $i\in[n]$, $h_i''(t_k)\to +\infty$ for any $\{t_k\}_{k\ge 0}\subseteq \inter\dom h_i$ such that $t_k\to t\in \bdry\dom h_i$. 
\end{enumerate}
With loss of generality, let $\dom h_i = [a_i,+\infty)$ for $i\in[n]$, and hence $\dom h:= \prod_{i=1}^n [a_i,+\infty)$. Define  $a:= (a_1,\ldots,a_n)$ and $\calD:= \dom h\setminus\{a\}$. 
Define $g:\bbR^n \to \barbbR$ such that 
\begin{align}
g(x) = \Bigg\{
\begin{aligned}
&\min_{i\in\calI(x)} h''(x_i), \quad &&  x\in \calD, \quad \mbox{for }\;\; \calI(x):= \{i\in[n]: x_i > a_i\}\ne \emptyset \\ 
&+\infty, \quad && x\not\in \calD  
\end{aligned}
\end{align}
Note that $\dom g = \calD$. 
Since $h_i''(t)>0$ for $t\in(a_i,+\infty)$, we have $g(x)>0$ for all $x\in \calD$.  
Next, we analyze the behavior of $g$ near $\bdry\dom h$.  Since for $i\in[n]$, $h_i''$ is continuous on $(a_i,+\infty)$ and $h_i''(t_k)\to +\infty$ for any $t_k\downarrow a_i$, for any $\{x^k\}_{k\ge 0}\subseteq \calD$ such that $x^k\to x\in \calD$, 
we have 
\begin{equation}
\lim_{k\to+\infty} g(x^k) = \lim_{k\to+\infty} \min_{i\in\calI(x^k)} h_i''(x^k_i) = \lim_{k\to+\infty} \min_{i\in\calI(x)} h_i''(x^k_i) = \min_{i\in\calI(x)} h_i''(x_i) = g(x).  \label{eq:conv_g}
\end{equation}
In addition, for any $\{x^k\}_{k\ge 0}\subseteq \calD$ such that $x^k\to a$, $g(x^k)\to +\infty$. This allows us to conclude that $g$ is a closed function. Indeed, let $\{(x^k,\tau_k)\}_{k\ge 0}\subseteq\epi g$ such that $(x^k,\tau_k)\to (x,\tau)$, where $\epi g$ denotes the epigraph of $g$.  Clearly $x\in\calD$, and by~\eqref{eq:conv_g},  we know that $g(x^k)\to g(x)$. Since $\tau_k\to \tau$ and $g(x^k)\le \tau_k$ for all $k\ge 0$, we have $g(x)\le \tau$. This shows that $(x,\tau)\in \epi g$.  

Next, note that for any $x\in \inter\calD$, we have $\nabla^2 h(x) = \Diag(h_1''(x_1),\ldots, h_n''(x_n))$, and hence   
\[
g(x) = {\min}_{i=1}^n\, h_i''(x_i) = {\min}_{\normt{z}_2=1}\, \ipt{\nabla^2 h(x) z}{z} = \alpha\, {\min}_{\normt{z}=1}\, \ipt{\nabla^2 h(x) z}{z} = \alpha\,\lambda_{\min}(\nabla^2 h(x)),
\]
where $\alpha>0$ is a constant independent of $x$. 
Now, fix any nonempty convex compact set $\calS\subseteq \dom h$ and any $z\in\inter\dom h$. Let $\calL_z:= \{x\in\bbX:g(x)\le g(z) \}\subseteq\calD$, which is nonempty and closed. Using the notations in Lemma~\ref{lem:suff_h}, we know that 
$\calS_z\cap\calL_z$ is nonempty and compact, and
\begin{align*}
\alpha\inf_{x\in \calS_z^o}\; \lambda_{\min} (\nabla^2 h(x)) = \inf_{x\in \calS_z^o}\; g(x)\gea \inf_{x\in \calS_z\cap\calD}\; g(x) \eqb \inf_{x\in \calS_z\cap\calD\cap\calL_z}\; g(x)
 \eqc \min_{x\in \calS_z\cap\calL_z}\; g(x) 
 \gsd 0,
\end{align*}
where (a) follows from $\inter\dom h\subseteq\calD$, (b) follows from $z\in \calS_z\cap \calD$, 
(c) follows from $\calL_z\subseteq\calD$ and that $\calS_z\cap\calL_z\ne \emptyset$ is  compact,  
and (d) follows from $\calS_z\cap\calL_z\subseteq\calD$ and $g(x)>0$ for all $x\in \calD$. Now, invoking Lemma~\ref{lem:suff_h}, we complete the proof. \qed

\section{Proof of Lemma~\ref{lem:seq} } \label{app:proof_seq}


Since for $k\ge k_0$, $a_k\le b_k$, we have 
\begin{equation}
a_{k+1}\le (1 - \tau_k) a_k + (A/2) \tau_k^2 \quad \Longrightarrow\quad  \beta_{k+1} a_{k+1}\le \beta_{k+1}(1 - \tau_k) a_k + (A/2) \beta_{k+1}\tau_k^2. \label{eq:seq0}
\end{equation}
By the choices of $\{\alpha_k\}_{k\ge 0}$ and $\{\beta_k\}_{k\ge 0}$ in~\ref{eq:step}, and that $\tau_k:= {\alpha_k }/\beta_{k+1}$,  we know that 
\begin{equation}
\beta_{k+1}(1 - \tau_{k}) = \beta_{k+1} - \alpha_k = \beta_k, \label{eq:seq1}
\end{equation}
and hence by~\eqref{eq:seq0}, we have
\begin{equation}
\beta_{k+1} a_{k+1}\le \beta_k a_k + (A/2) \beta_{k+1}\tau_k^2. \label{eq:seq_1.5}
\end{equation} 
For $k\ge k_0+1$, telescope~\eqref{eq:seq_1.5} over $i=k_0,\ldots,k-1$, and we have 
\begin{align}
a_k&\le \frac{\beta_{k_0}a_{k_0} + (A/2)\sum_{i=k_0}^{k-1}\beta_{i+1} \tau_i^2}{\beta_k } = \frac{\beta_{k_0}a_{k_0} + (A/2)\sum_{i=k_0}^{k-1} \alpha_i^2/\beta_{i+1}  }{\beta_{k} }. \label{eq:seq_1.8}
\end{align}
Substitute the choices of $\{\alpha_k\}_{k\ge 0}$ and $\{\beta_k\}_{k\ge 0}$ in~\ref{eq:step} into~\eqref{eq:seq_1.8}, and we arrive at~\eqref{eq:a_k}.  
Now, telescope the second inequality in~\eqref{eq:recur} over $i=\bark,\ldots,k-1$ for some $k_0\le \bark \le k-1$, and we have
\begin{align}
\textstyle 0\le a_k \le a_{\bark} - \sum_{i=\bark}^{k-1}\, \tau_i b_i + (A/2) \sum_{i=\bark}^{k-1}\,\tau_i^2. \label{eq:seq3}
\end{align}
Since $\tau_k = 2/(k+2)$ for $k\ge k_0$, we have 
\begin{align}
\textstyle &\sum_{i=\bark}^{k-1}\,\tau_i  \ge (k - \bark) \tau_{k-1} = \dfrac{2(k - \bark)}{k+1}, \quad \andd\\
\textstyle & \sum_{i=\bark}^{k-1}\,\tau_i^2 \le  4\sum_{i=\bark}^{k-1}\,\frac{1}{(i+1)(i+2)} = 4\left(\frac{1}{\bark+1} - \frac{1}{k+1}\right) = \frac{4(k-\bark)}{(\bark+1)(k+1)}.
\end{align}
As a result, from~\eqref{eq:seq3}, we have 
\begin{align}
\min_{i=\bark, \ldots, k-1}\; b_i \le \frac{ a_{\bark}+ (A/2) \sum_{i=\bark}^{k-1}\,\tau_i^2}{\sum_{i=\bark}^{k-1}\,\tau_i}\le \frac{k+1}{2(k - \bark)}a_{\bark} + \frac{A}{\bark+1}.  \label{eq:seq_min_b} 
\end{align}
Let $\bark = \lfloor (k+k_0)/2 \rfloor\ge k_0$, so that 
\begin{equation}
\frac{k+k_0-1}{2} \le \bark \le \frac{k+k_0}{2} \le \bark + 1. \label{eq:seq_floor}
\end{equation} 
If $k\ge k_0 + 2 $, then $\bark\ge k_0 + 1$, and from~\eqref{eq:a_k} and~\eqref{eq:seq_floor}, we have 
\begin{align}
\frac{k+1}{2(k - \bark)}a_{\bark}  &\le \frac{k+1}{2(k - \bark)}\cdot\frac{{k_0}(k_0+1)a_{k_0} + 2A(\bark-k_0)}{\bark(\bark+1)}\\
&\le \frac{k+1}{k - k_0}\cdot\frac{{k_0}(k_0+1)a_{k_0} + 2A(k-k_0)}{(k+k_0-1)(k+k_0)/4} \label{eq:seq_3.5} \\
&\le \frac{12({k_0}(k_0+1)a_{k_0} + 2A(k-k_0))}{(k - k_0)(k+k_0)},  \label{eq:seq4}
\end{align}
where~\eqref{eq:seq_3.5} follows from~\eqref{eq:seq_floor}  and~\eqref{eq:seq4} follows from 
\begin{equation}
\frac{k+1}{k+k_0-1}\le \frac{k+1}{k-1} \le 3, \quad \forall\,k\ge 2. 
\end{equation}
In addition, we have ${A}/({\bark+1})\le 2A/(k+k_0)$, and so from~\eqref{eq:seq_min_b}, we have 
\begin{align}
\min_{i=\lfloor (k+k_0)/2 \rfloor, \ldots, k-1}\; b_i \le
 \frac{12(k_0+1)^2 }{(k - k_0)(k+k_0)} a_{k_0}+ 
\frac{26A}{k+k_0}, \quad \forall\, k\ge k_0 + 2. \label{eq:seq_final}  
\end{align}
Lastly, note that if $k= k_0 + 1 $, then $\bark= k_0 $, and from~\eqref{eq:seq_min_b}, we have
\begin{align}
b_{k_0}\le \left(\frac{k_0}{2}+1\right)a_{k_0} + \frac{A}{k_0+1}\le \frac{12(k_0+1)^2}{2k_0+1}a_{k_0} + \frac{26A}{2k_0+1},  \label{eq:seq_final1}
\end{align}
and the second inequality precisely corresponds to the right-hand side of~\eqref{eq:seq_final} when $k= k_0 + 1 $.  
Combining~\eqref{eq:seq_final} and~\eqref{eq:seq_final1}, we arrive at~\eqref{eq:b_k}. \qed

\section{Proof of Lemma~\ref{lem:Delta} } \label{app:proof_Delta}

Let us first prove a more general result. 

\begin{lemma} \label{lem:dist_AB}
Let $\bbU:=(\bbR^d, \normt{\cdot})$ be a normed space,  $\emptyset\ne \calA\subseteq \bbU$ be  compact, and $\emptyset\ne \calB\subseteq \bbU$ be closed. Then there exist $a\in\calA$ and $b\in\calB$ such that $\dist_{\normt{\cdot}}(\calA,\calB) = \normt{a-b}$. 
\end{lemma}

\begin{proof}[Proof of Lemma~\ref{lem:Delta}.]
Based on Lemma~\ref{lem:dist_AB}, 
we simply note that $\emptyset\ne \bar\calU\subseteq \dom h^*$ is compact and $\bdry\dom h^*\ne\emptyset$ is closed, 
and $\bar\calU\cap \bdry\dom h^*
=\emptyset$ under Assumption~\ref{assum:open_dom}. 
\end{proof}

\begin{proof}[Proof of Lemma~\ref{lem:dist_AB}]
Consider the proper and closed function
\begin{equation}
\Phi(u,u'):= \normt{u-u'} + \iota_{\calA}(u) + \iota_{\calB}(u'), \quad \forall\, u,u'\in \bbU,
\end{equation}
such that $\inf_{u,u'\in\calU}\; \Phi(u,u') = \dist_{\normt{\cdot}}(\calA,\calB)$. 
Note that $\Phi$ is coercive: indeed, take any $r\ge 0$, then the $r$-sub-level set of $\Phi$, namely
\begin{equation}
\calL_r:=\{(u,u')\in \bbU\times \bbU:\Phi(u,u')\le r \} = \{u\in\calA,u'\in\calB:\normt{u-u'}\le r \},
\end{equation}
is clearly bounded. Since $\Phi$ is proper, closed and coercive, it has a minimizer $(a,b)\in \calA\times \calB$ and hence $\dist_{\normt{\cdot}}(\calA,\calB) = \inf_{u,u'\in\calU}\; \Phi(u,u') = \normt{a-b} $. 
\end{proof}

\section{Proof of Lemma~\ref{lem:Ur} } \label{app:proof_Ur}

The proof of Lemma~\ref{lem:Ur} relies on the following three lemmas.

\begin{lemma} \label{lem:dist_func}
Let $\bbU:=(\bbR^d, \normt{\cdot})$ be a normed space and  $\emptyset\ne \calA\subseteq \bbU$. Then the distance  function $\dist_{\normt{\cdot}}(\cdot,\calA):\bbU\to \bbR$ is 1-Lipschitz on $\bbU$. In addition, it is convex if $\calA$ is convex. 
\end{lemma}

\begin{lemma} \label{lem:pass_through_boundary}
Let $\bbU:=(\bbR^d, \normt{\cdot})$ be a normed space and $\emptyset\ne \calA\subsetneqq \bbU$. For any $u\in\calA$ and $u\in\calA^c$, define $[u,u']:= \conv(\{u,u'\})$. Then $[u,u']\cap\bdry\calA\ne\emptyset$. 
\end{lemma}

\begin{lemma} \label{lem:u_notin_domh*}
If $u\not\in \dom h^*$, then $\dist_{\normt{\cdot}_*}(u,\bar\calU)\ge \Delta$.
\end{lemma}

\begin{proof}[Proof of Lemma~\ref{lem:Ur}]
First note that $\bar\calU(r)$ is nonempty and bounded, since $\bar\calU$ is nonempty and bounded. By Lemma~\ref{lem:dist_func}, we know that $\bar\calU(r)$ is closed and convex. Therefore, $\bar\calU(r)$ is nonempty, convex and compact. Suppose that  $\bar\calU(r)\not\subseteq\dom h^*$ for some $0\le r<\Delta$, then there exists $u\in \bar\calU(r)$ such that $u\not\in \dom h^*$. By Lemma~\ref{lem:u_notin_domh*}, we have $\dist_{\normt{\cdot}_*}(u,\bar\calU)\ge \Delta$. 
However, since $u\in \bar\calU(r)$  and $r<\Delta$, we know that $\dist_{\normt{\cdot}_*}(u,\bar\calU)\le r< \Delta$. This leads to a contradiction.
\end{proof}


\begin{proof}[Proof of Lemma~\ref{lem:dist_func}]
For convenience, we omit the subscript $\normt{\cdot}$ in the distance function. 
Fix any $u,u'\in\bbU$. 
For any $a\in\calA$, we have 
\begin{equation}
\dist(u,\calA)\le \normt{u-a}\le \normt{u-u'} + \normt{u'-a},
\end{equation}
and hence $\dist(u,\calA)\le \normt{u-u'} + \dist(u',\calA)$, or equivalently,  $\dist(u,\calA) - \dist(u',\calA)\le \normt{u-u'} $. By swapping the role of $u$ and $u'$, we easily see that 
$\abst{\dist(u,\calA) - \dist(u',\calA)}\le \normt{u-u'} $. Now, suppose that $\calA$ is convex. For any $\epsilon>0$, there exist $a,a'\in\calA$ such that $\normt{u-a}\le \dist(u,\calA)+\epsilon$ and $\normt{u'-a'}\le \dist(u',\calA)+\epsilon$. Let us fix any $\lambda\in[0,1]$.  Since $\lambda a + (1-\lambda)a'\in\calA$, we have
\begin{align}
\dist(\lambda u + (1-\lambda)u',\calA) &\le \normt{(\lambda u + (1-\lambda)u') - (\lambda a + (1-\lambda)a')}\\
&\le \lambda \normt{u - a} + (1-\lambda)\normt{u' - a'}\\
& \le \lambda\dist(u,\calA) + (1-\lambda)\dist(u',\calA) + \epsilon. 
\end{align}
By letting $\epsilon\to 0$, we finish the proof. 
\end{proof}

\begin{proof}[Proof of  Lemma~\ref{lem:pass_through_boundary} ]
If $\inter\calA=\emptyset$ or $\inter\calA^c=\emptyset$, since $\bdry\calA = \bdry\calA^c$, we know that either $u$ or $u'$ lies in $\bdry\calA$, and the lemma trivially holds. Therefore, we focus on the case where both $\inter\calA$ and $\inter\calA^c$ are nonempty. If $[u,u']\cap\bdry\calA=\emptyset$, then $[u,u']$ is separated by  $\inter\calA$ and $\inter\calA^c$, namely, $\inter\calA\cap\inter\calA^c = \emptyset$, $\inter\calA\cap[u,u'] \ne \emptyset$, $\inter\calA^c\cap[u,u'] \ne \emptyset$ and $[u,u']\subseteq \inter\calA\cup\inter\calA^c$, and hence is disconnected. However,  $[u,u']$ is clearly path-connected, and hence connected. This leads to a contradiction. 
\end{proof}

\begin{proof}[Proof of Lemma~\ref{lem:u_notin_domh*}]
For any $u'\in \bar\calU\subseteq\dom h^*$, by Lemma~\ref{lem:pass_through_boundary}, there exists $\tilu\in[u,u']$ such that $\tilu \in \bdry \dom h^*$. Write $\tilu = \lambda u +(1-\lambda) u'$ for some $\lambda\in[0,1]$, and we have $\normt{\tilu-u'}_* = \lambda\normt{u-u'}_* \le \normt{u-u'}_*$, and hence $\Delta = \dist_{\normt{\cdot}_*}(\bdry \dom h^*,\bar\calU)\le \dist_{\normt{\cdot}_*}(\tilu,\bar\calU) \le \dist_{\normt{\cdot}_*}(u,\bar\calU). $
%
\end{proof}

\section{Proof of Lemma~\ref{lem:diamU} } \label{app:proof_diamU}

The proof of Lemma~\ref{lem:diamU} hinges upon the following lemma.

\begin{lemma} \label{lem:sup_gauge}
Let $\calC$ be nonempty and bounded such that $\calC\ne \{0\}$. Then $\sup_{x\in\calC}\, \gamma_\calC(x) = 1$, where $\gamma_\calC$ denotes the gauge function of $\calC$. 
\end{lemma}

\begin{proof}
For convenience, define $\zeta:= \sup_{x\in\calC}\, \gamma_\calC(x)$. 
Since  $\calC\ne \emptyset$ and  $\calC\ne \{0\}$, there exists $u\ne 0$ such that $u\in\calC$. Since $\calC$ is bounded, $\gamma_\calC(u)>0$, and there exists a positive sequence $\{\lambda_k\}_{k\ge 0}$ such that $\lambda_k\downarrow \gamma_\calC(u)$ and $u/\lambda_k\in \calC$ for all $k\ge 0$. Since $\gamma_\calC$ is positively homogeneous  (by definition),  we have $\zeta \ge \gamma_\calC(u/\lambda_k) = \gamma_\calC(u)/\lambda_k$ for $k\ge 0$. By taking limit, we have $\zeta\ge 1$. On the other hand, by definition, we have $\gamma_\calC(x)\le 1$ for all $x\in\calC$, and hence $\zeta\le 1$. This completes the proof. 
\end{proof}

Now, since $\calU$ is solid and compact, so is $\calU - \calU$. 
By Lemma~\ref{lem:sup_gauge},  we have 
\begin{align*}
\diam_{\normt{\cdot}_{\calU}}(\calU) &= {\max}_{u,u'\in\calU}\, \normt{u-u'}_\calU\\
 &= {\max}_{u,u'\in\calU}\; \gamma_{\calU - \calU}(u-u') = {\max}_{v\in\calU- \calU}\; \gamma_{\calU - \calU}(v) = 1. 
\end{align*}

\section{Proof of Proposition~\ref{lem:Lipschitz_extension} }\label{app:proof_Lipschitz_extension}

We first show that $F_L = f$ on $\calC$. Indeed, for any $z\in\calC$, by taking  $z'=z$ in~\eqref{eq:PH_env}, we have $F_L(z)\le f(z)$. On the other hand, for any $z,z'\in \calC$, since $f$ is $L$-Lipschitz on $\calC$,  we have 
\begin{equation}
f(z') + L\normt{z-z'}_*\ge f(z)\quad \Longrightarrow\quad F_L(z) = {\inf}_{z'\in \calC}\;  f(z') + L\normt{z-z'}_*\ge f(z).
\end{equation} 
Next, note that for any $z,v\in\bbY^*$, we have
\begin{align}
F_L(z)\le \textstyle \inf_{z'\in \bbY^*}\;  f_\calC(z') + L\normt{v-z'}_* + L\normt{z-v}_*  = F_L(v) + L\normt{z-v}_*. 
\end{align}
Therefore, $F_L$ is real-valued on $\bbY^*$ and $F_L(z)-F_L(v) \le  L\normt{z-v}_*$ for all $z,v\in\bbY^*$.  
This implies that $F_L$ is $L$-Lipschitz on $\bbY^*$. Finally, the convexity of $F_L$ follows from 
the joint convexity of the function $(z,z')\mapsto f_\calC(z') + L\normt{z-z'}_*$ on $\bbY^*\times\bbY^*$ (see e.g.,~\cite[Prop.~8.26]{Bauschke_11}). \qed

\section{Proof of Proposition~\ref{prop: partial_FL}} \label{app:proof_partial_FL}

 Since $f_\calC$ is proper  and convex and $(L\normt{\cdot}_{*})^* = \iota_{\calB_{\normt{\cdot}}(0,L)}$,  by~\cite[Theorem~16.4]{Rock_70}, we have 
\begin{equation}
F_L^*= (f_\calC\,\square\,L\normt{\cdot}_{*})^*=  f_\calC^* + \iota_{\calB_{\normt{\cdot}}(0,L)}. \quad \label{eq:FL_conj}
\end{equation}
 In addition, since $F_L$ is proper, closed and convex, we have for all $z\in\bbY^*, $
\begin{align}
F_L(z) = {\sup}_{y\in\bbY}\;  \ipt{u}{y} - F_L^*(y) = {\sup}_{\normt{y}\le L}\;  \big\{\psi_z(y):=\ipt{z}{y} - f_\calC^*(y)\big\}, 
\label{eq:dual_FL}
\end{align}
and $\partial F_L(z)= {\argmax}_{\normt{y}\le L}\, \psi_z(y)$. As a result, we have $\partial F_L(z)\subseteq \calB_{\normt{\cdot}}(0,L)$.
Note that since $F_L$ is globally convex and Lipschitz, $\partial F_L(z)\ne \emptyset$ for all $z\in\bbY^*$. 
Now, fix any $z\in\calC$. 
For all $g\in\partial F_L(z)$, 
\begin{align}
f_\calC(z') = F_L(z')\ge F_L(z) + \ipt{g}{z'-z} = f_\calC(z) + \ipt{g}{z'-z}, \quad\forall\,z'\in\calC,
\end{align}
and therefore $g\in \partial f_\calC(z)$. 
Thus 
we have $\partial F_L(z) \subseteq \partial f_\calC(z)$, 
 and hence 
$\partial F_L(z)\subseteq \partial f_\calC(z)\cap\calB_{\normt{\cdot}}(0,L)$. 
Next, take any $g\in \partial f_\calC(z)$ 
such that $\normt{g}\le L$, and we have 
\begin{equation}
F_L(z) \ge \psi_z(g) = \ipt{z}{g} - f_\calC^*(g) \eqa f_\calC(z) = F_L(z),  
\end{equation}
where (a) follows from~\cite[Theorem~23.5]{Rock_70} and that  $f_\calC$ is proper  and convex. 
As a result,  $F_L(z) = \psi_z(g)$ and hence $g\in \argmax_{\normt{y}\le L}\, \psi_z(y) $, 
which then implies that $g\in \partial F_L(z)$. Lastly, note that from~\cite[Theorem~23.8]{Rock_70}, we know that for all $z\in \bbY^*$, $\partial f(z) + \calN_\calC(z) \subseteq   \partial(f+\iota_\calC)(z) = \partial f_\calC(z)$, with equality holds if $\ri\dom f \cap\ri\calC\ne \emptyset$. 

\bibliographystyle{IEEEtr}
\bibliography{math_opt,mach_learn}

\end{document}

%% file: preamble3.tex
\usepackage[mathscr]{eucal}
\usepackage{epsfig,epsf,psfrag}
\usepackage{amssymb,amsfonts,latexsym}
\usepackage{amsmath}
\usepackage{graphicx}
\usepackage{epstopdf}
\usepackage{bm,xcolor,url}
\usepackage{fixltx2e}
\usepackage{array}
\usepackage{verbatim}
\usepackage{algpseudocode}
\usepackage[noadjust]{cite}
\usepackage{algorithm}
\usepackage{verbatim}
\usepackage{textcomp}
\usepackage{mathrsfs}
\usepackage[referable]{threeparttablex}
\usepackage{subfig}
\usepackage{amsthm}
\usepackage{enumerate}
\usepackage{footnote}
\usepackage{enumitem}
\usepackage{xr}
\usepackage{mathtools}
\catcode`~=11 \def\UrlSpecials{\do\~{\kern -.15em\lower .7ex\hbox{~}\kern .04em}} \catcode`~=13 

\allowdisplaybreaks[3]

\newcommand{\tnorm}[1]{{\left\vert\kern-0.25ex\left\vert\kern-0.25ex\left\vert #1 
    \right\vert\kern-0.25ex\right\vert\kern-0.25ex\right\vert}}
\newcommand{\tnormt}[1]{{\vert\kern-0.25ex\vert\kern-0.25ex\vert #1 
    \vert\kern-0.25ex\vert\kern-0.25ex\vert}}

\newcommand{\where}{\;\; \mbox{where}\;\;}
\newcommand{\andd}{\;\; \mbox{and}\;\;}
\newcommand{\ipt}{\lranglet}

\newcommand{\cone}{\mathsf{cone}\,}

\newcommand{\normt}[1]{\Vert#1\Vert}
\newcommand{\abst}[1]{\vert#1\vert}

\newcommand{\nn}{\nonumber}

\newcommand{\prox}{\mathsf{prox}}
\newcommand{\dom}{\mathsf{dom}\,}

\newcommand{\inter}{\mathsf{int}\,}
\newcommand{\bdry}{\mathsf{bd}\,}

\newcommand{\Diag}{\mathsf{Diag}\,}
\newcommand{\cl}{\mathsf{cl}\,}
\newcommand{\ri}{\mathsf{ri}\,}

\newcommand{\aff}{\mathsf{aff}\,}
\newcommand{\lin}{\mathsf{lin}\,}
\newcommand{\epi}{\mathsf{epi}\,}
\newcommand{\conv}{\mathsf{conv}\,}

\newcommand{\diam}{\mathsf{diam}\,}
\newcommand{\ran}{\mathsf{ran}\,}

\newcommand{\dist}{\mathsf{dist}\,}

\newcommand{\calA}{\mathcal{A}}
\newcommand{\calB}{\mathcal{B}}
\newcommand{\calC}{\mathcal{C}}
\newcommand{\calD}{\mathcal{D}}

\newcommand{\calI}{\mathcal{I}}

\newcommand{\calK}{\mathcal{K}}
\newcommand{\calL}{\mathcal{L}}

\newcommand{\calN}{\mathcal{N}}
\newcommand{\calO}{\mathcal{O}}
\newcommand{\calP}{\mathcal{P}}
\newcommand{\calQ}{\mathcal{Q}}
\newcommand{\calR}{\mathcal{R}}
\newcommand{\calS}{\mathcal{S}}

\newcommand{\calU}{\mathcal{U}}
\newcommand{\calV}{\mathcal{V}}
\newcommand{\calW}{\mathcal{W}}
\newcommand{\calX}{\mathcal{X}}

\newcommand{\tilcalQ}{\widetilde{\calQ}} 
 
\newcommand{\tilcalS}{\widetilde{\calS}} 
 
\newcommand{\tilcalU}{\widetilde{\calU}}

\newcommand{\barcalA}{\bar{\calA}}

\newcommand{\barcalL}{\bar{\calL}}

\newcommand{\barcalS}{\bar{\calS}} 
 
\newcommand{\barcalU}{\bar{\calU}}

\newcommand{\rmP}{\mathrm{P}}

\newcommand{\bbR}{\mathbb{R}}

\newcommand{\bbU}{\mathbb{U}}

\newcommand{\bbW}{\mathbb{W}}
\newcommand{\bbX}{\mathbb{X}}
\newcommand{\bbY}{\mathbb{Y}}

\newcommand{\barbbR}{\overline{\bbR}}

\newcommand{\scA}{\mathscr{A}}

\DeclareMathAlphabet{\mathbsf}{OT1}{cmss}{bx}{n}

\newcommand{\rvA}{\mathsf{A}}

\newcommand{\rvM}{\mathsf{M}}

\newcommand{\rvT}{\mathsf{T}}

\newcommand{\tilf}{\tilde{f}}

\newcommand{\tilh}{\tilde{h}}

\newcommand{\hats}{\hat{s}}

\newcommand{\tilu}{\tilde{u}}

\newcommand{\hatx}{\hat{x}}

\newcommand{\tilx}{\widetilde{x}}

\newcommand{\bark}{\bar{k}}

\newcommand{\bars}{\bar{s}}

\newcommand{\baru}{\bar{u}}

\newcommand{\barx}{\bar{x}}

\newcommand{\lranglet}[2]{\langle{#1},{#2}\rangle}

\newcommand{\eqa}{\stackrel{\rm(a)}{=}}
\newcommand{\eqb}{\stackrel{\rm(b)}{=}}
\newcommand{\eqc}{\stackrel{\rm(c)}{=}}

\newcommand{\lea}{\stackrel{\rm(a)}{\le}}
\newcommand{\leb}{\stackrel{\rm(b)}{\le}}

\newcommand{\gea}{\stackrel{\rm(a)}{\ge}}
\newcommand{\geb}{\stackrel{\rm(b)}{\ge}}
\newcommand{\gec}{\stackrel{\rm(c)}{\ge}}

\newcommand{\gsd}{\stackrel{\rm(d)}{>}}

\DeclareMathOperator*{\argmax}{arg\,max}
\DeclareMathOperator*{\argmin}{arg\,min}

\newtheorem{theorem}{Theorem} 
\newtheorem{lemma}{Lemma}
\newtheorem{prop}{Proposition}

\theoremstyle{definition}
\newtheorem{example}{Example} 
\newtheorem{assump}{Assumption} 
\newtheorem{definition}{Definition} 
\theoremstyle{remark}

\newtheorem{remark}{Remark}

\newcommand{\qednew}{\nobreak \ifvmode \relax \else
      \ifdim\lastskip<1.5em \hskip-\lastskip
      \hskip1.5em plus0em minus0.5em \fi \nobreak
      \vrule height0.75em width0.5em depth0.25em\fi}